\documentclass[11pt,letterpaper]{amsart}

\usepackage[margin=1.08in]{geometry}
\usepackage{amsmath,amssymb,amsthm,mathtools}
\usepackage{microtype}
\usepackage[T1]{fontenc}
\usepackage{libertinus}
\usepackage{enumitem}
\usepackage{booktabs}
\usepackage{tabularx}
\usepackage[dvipsnames]{xcolor}
\usepackage{array}
\usepackage{float}
\usepackage{tikz}
\usetikzlibrary{patterns,calc,positioning,arrows.meta}
\usepackage{cite}
\usepackage{hyperref}
\usepackage{aliascnt}
\usepackage[nameinlink,capitalize]{cleveref}

\hypersetup{
  colorlinks=true,
  linkcolor=black,
  citecolor=black,
  urlcolor=black,
  pdftitle={Sparse KSZ and Weighted Hardy--Littlewood Inequalities},
  pdfauthor={Anderson Barbosa, Daniel Nunez-Alarcon, Anselmo Raposo, and Joedson Santos}
}
\setlist{itemsep=0.25em,topsep=0.45em}
\allowdisplaybreaks
\numberwithin{equation}{section}
\definecolor{gmmGold}{HTML}{F2C94C}
\definecolor{gmmRose}{HTML}{C96B8D}
\definecolor{gmmCoral}{HTML}{F2997A}
\definecolor{gmmIndigo}{HTML}{6C75D8}
\definecolor{gmmOlive}{HTML}{B7B86A}
\definecolor{gmmViolet}{HTML}{8B6FD6}
\definecolor{ourTeal}{HTML}{169C95}
\definecolor{axisInk}{HTML}{27313A}
\colorlet{gmmGold}{black!16}
\colorlet{gmmRose}{black!24}
\colorlet{gmmCoral}{black!30}
\colorlet{gmmIndigo}{black!38}
\colorlet{gmmOlive}{black!45}
\colorlet{gmmViolet}{black!52}
\colorlet{ourTeal}{black!68}
\colorlet{axisInk}{black}

\newtheorem{theorem}{Theorem}[section]

\newaliascnt{proposition}{theorem}
\newtheorem{proposition}[proposition]{Proposition}
\aliascntresetthe{proposition}

\newaliascnt{lemma}{theorem}
\newtheorem{lemma}[lemma]{Lemma}
\aliascntresetthe{lemma}

\newaliascnt{corollary}{theorem}
\newtheorem{corollary}[corollary]{Corollary}
\aliascntresetthe{corollary}

\theoremstyle{definition}
\newaliascnt{definition}{theorem}
\newtheorem{definition}[definition]{Definition}
\aliascntresetthe{definition}

\theoremstyle{remark}
\newaliascnt{remark}{theorem}
\newtheorem{remark}[remark]{Remark}
\aliascntresetthe{remark}

\crefname{theorem}{theorem}{theorems}
\Crefname{theorem}{Theorem}{Theorems}
\crefname{proposition}{proposition}{propositions}
\Crefname{proposition}{Proposition}{Propositions}
\crefname{lemma}{lemma}{lemmas}
\Crefname{lemma}{Lemma}{Lemmas}
\crefname{corollary}{corollary}{corollaries}
\Crefname{corollary}{Corollary}{Corollaries}
\crefname{definition}{definition}{definitions}
\Crefname{definition}{Definition}{Definitions}
\crefname{remark}{remark}{remarks}
\Crefname{remark}{Remark}{Remarks}

\newcommand{\R}{\mathbb R}
\newcommand{\C}{\mathbb C}
\newcommand{\K}{\mathbb K}
\newcommand{\N}{\mathbb N}
\newcommand{\T}{\mathbb T}
\newcommand{\cK}{\mathcal K}
\newcommand{\sK}{\mathsf K}
\newcommand{\cL}{\mathcal L}
\newcommand{\eps}{\varepsilon}
\newcommand{\pos}[1]{\bigl(#1\bigr)_{+}}

\newcommand{\Wfun}{\mathcal W}
\newcommand{\Nzero}{\mathbb N_0}
\newcommand{\supp}{\operatorname{supp}}

\title[Sparse KSZ and weighted Hardy--Littlewood inequalities]
{Sparse Kahane--Salem--Zygmund Forms and Weighted Hardy--Littlewood Inequalities Across the Critical Endpoint}

\author[A. Barbosa]{Anderson Barbosa}
\address{Universidade Federal Rural do Semi-\'Arido (UFERSA), Rio Grande do Norte, Brazil}
\email{anderson.barbosa@ufersa.edu.br}

\author[D. N\'u\~nez-Alarc\'on]{Daniel N\'u\~nez-Alarc\'on}
\address{Departamento de Matem\'aticas, Facultad de Ciencias, Universidad Nacional de Colombia, Sede Bogot\'a, Bogot\'a, Colombia}
\email{dnuneza@unal.edu.co}

\author[A. Raposo]{Anselmo Raposo}
\address{Departamento de Matem\'atica, Universidade Federal do Maranh\~ao, S\~ao Lu\'is, MA, Brazil}
\email{anselmo.junior@ufma.br}

\author[J. Santos]{Joedson Santos}
\address{Departamento de Matem\'atica, Centro de Ci\^encias Exatas e da Natureza, Universidade Federal da Para\'iba, Jo\~ao Pessoa, PB, Brazil}
\email{joedson.santos@academico.ufpb.br}

\date{}

\begin{document}

\begin{abstract}
We study sparse Kahane--Salem--Zygmund constructions and weighted Hardy--Littlewood inequalities for homogeneous polynomials. For supports of cardinality $n^{d+o(1)}$, we determine the sharp power of $n$ governing the smallest norm of a unimodular $m$-linear form on $\ell_{p_1}^n\times\cdots\times\ell_{p_m}^n$; in the diagonal case, this yields the missing polynomial growth exponent in the coefficient-versus-supremum norm problem for $2\le p\le m$ and $2\le r\le\infty$. We then introduce a diagonal weighted Hardy--Littlewood functional which, on $s=p\ge m$, agrees exactly with the classical Hardy--Littlewood coefficient norm with the same optimal constant. We determine the optimal diagonal weight exponent for $2\le p\le m$ and $1\le q\le2$, on the full critical line $p=m$, and on a sharp part of the region $q>2$; at $q=\infty$ the optimal weight exponent is obtained for every $2\le p\le m$. The sparse coefficient estimates provide the matching dimensional obstructions.
\end{abstract}

\maketitle
\setcounter{tocdepth}{1}

\section{Introduction}\label{sec:introduction}

For an $m$-homogeneous polynomial
\[
P(z)=\sum_{|\alpha|=m}a_\alpha z^\alpha
\qquad (z\in\ell_p^n),
\]
the relation between its polynomial norm and the norm of its coefficient
sequence is expressed by inequalities of the form
\[
\left\|(a_\alpha)_{|\alpha|=m}\right\|_{\ell_r}
\le C_{m,p,r}(n)
\|P\|_{\mathcal P({}^m\ell_p^n)}.
\]
The dimension-free cases include the Bohnenblust--Hille inequality on
$\ell_\infty^n$ \cite{BohnenblustHille1931,DefantEtAl2011} and the
polynomial Hardy--Littlewood inequalities on $\ell_p^n$
\cite{HardyLittlewood1934,PracianoPereira,AraujoEtAl2015,
AronNunezPellegrinoSerrano}. When the optimal constant depends on $n$,
its growth measures the norm of the finite-dimensional coefficient
embedding. Estimates for these embedding constants in several parameter
ranges were obtained in \cite{GalicerMansillaMuro}.

Lower bounds for coefficient embeddings can be obtained from
multilinear forms with unimodular nonzero coefficients. For
$\mathbf p=(p_1,\ldots,p_m)\in[2,\infty]^m$ and a support
$\Lambda\subset[n]^m$, consider
\[
T_{\Lambda,\varepsilon}(x^{(1)},\ldots,x^{(m)})
=
\sum_{\mathbf i\in\Lambda}
\varepsilon_{\mathbf i}
x^{(1)}_{i_1}\cdots x^{(m)}_{i_m},
\qquad |\varepsilon_{\mathbf i}|=1.
\]
The infimum of the operator norm of these forms under
$|\Lambda|\asymp n^d$ depends on the cardinality of $\Lambda$ and on
the norms of the coordinate restrictions of
$T_{\Lambda,\varepsilon}$. For full supports, norm estimates of this
type belong to the Kahane--Salem--Zygmund theory
\cite{Kahane1985,Boas2000,DefantMastylo2023,
AlbuquerqueRezende,PellegrinoSerranoSilva,PellegrinoRaposo}.

For prescribed supports, the coordinate projections of rectangular
restrictions determine the dimensions of the subspaces on which the
restricted forms act. Related support constructions and coefficient
estimates occur in \cite{Blei1984,BleiKorner1984,BayartJEMS}.
In particular, \cite[Corollary~4.12]{BayartJEMS} gives unimodular
multilinear forms with supports of prescribed polynomial cardinality
in several parameter ranges. Partial Steiner supports yield
homogeneous polynomials with approximately $n^{m-1}$ nonzero
unimodular coefficients and controlled norm
\cite{GalicerMuroSevillaPeris2018,GalicerMansillaMuro}.

For isotropic exponents, the zero-growth threshold for unimodular
forms supported on sets of cardinality comparable to $n^d$, together
with the associated multilinear summability invariant, was determined
in \cite{BarbosaPellegrinoRaposoTeixeira2026}. The local estimates
used here concern rectangular restrictions and their coordinate
projection sizes. Their anisotropic form retains the dependence on
each exponent $p_j$.

Throughout, $[m]:=\{1,\ldots,m\}$ and $x_+:=\max\{x,0\}$ for $x\in\mathbb R$.

Let $\mathbf p=(p_1,\ldots,p_m)\in[2,\infty]^m$, and put
\[
 \sigma:=\sum_{j=1}^m\frac1{p_j},
 \qquad
 \sigma_I:=\sum_{j\in I}\frac1{p_j}.
\]
Let
\[
 \mathcal I(\mathbf p):=
 \left\{\varnothing\ne I\subset[m]:
 \frac12\le\sigma_I<1\right\}.
\]
For $d\in[1,m]$ define
\begin{equation}\label{eq:alpha-aniso}
 \alpha_{m,d}(\mathbf p):=
 \begin{cases}
 \displaystyle
 d\,\frac{m+1-2\sigma}{2m},
 &\sigma\le\frac12,\\[2mm]
 \displaystyle
 \max\left\{
 0,\frac{d+1}{2}-\sigma,
 \max_{I\in\mathcal I(\mathbf p)}
 (d-m+|I|)_+(1-\sigma_I)
 \right\},
 &\sigma\ge\frac12.
 \end{cases}
\end{equation}

When $\sigma\ge\frac12$, the three quantities in the maximum are, respectively, the basis-vector obstruction, the global $\ell_2$ obstruction, and the obstructions produced by coordinate sections.

For $\Lambda\subset[n]^m$ and scalars
$\varepsilon_{\mathbf i}\in\mathbb K$ satisfying
$|\varepsilon_{\mathbf i}|=1$, define the sparse unimodular form
\[
 T_{\Lambda,\varepsilon}
 (x^{(1)},\ldots,x^{(m)})
 :=
 \sum_{\mathbf i=(i_1,\ldots,i_m)\in\Lambda}
 \varepsilon_{\mathbf i}
 x^{(1)}_{i_1}\cdots x^{(m)}_{i_m}.
\]
For $1\le N\le n^m$, set
\[
 \sK^{\mathbb K}_{m,\mathbf p}(n,N)
 :=
 \inf\left\{
 \|T_{\Lambda,\varepsilon}\|:
 \Lambda\subset[n]^m,\;
 |\Lambda|\ge N,\;
 |\varepsilon_{\mathbf i}|=1
 \text{ for every }\mathbf i\in\Lambda
 \right\},
\]
where the norm is taken on
\[
 \ell_{p_1}^n\times\cdots\times\ell_{p_m}^n.
\]
Finally, put
\[
 \cK^{\mathbb K}_{m,d,\mathbf p}(n)
 :=
 \sK^{\mathbb K}_{m,\mathbf p}
 (n,\lfloor n^d\rfloor).
\]
\begin{theorem}\label{thm:A}
Let $m\ge2$, $d\in[1,m]$, $\mathbf p=(p_1,\ldots,p_m)\in[2,\infty]^m$, and $\K\in\{\R,\C\}$. Then
\begin{equation}\label{eq:main-asymptotic}
 \cK^{\K}_{m,d,\mathbf p}(n)
 =n^{\alpha_{m,d}(\mathbf p)+o(1)}.
\end{equation}
More precisely, there is $c=c(m,d,\mathbf p)>0$ such that
\begin{equation}\label{eq:main-lower}
 \cK^{\K}_{m,d,\mathbf p}(n)\ge c\,n^{\alpha_{m,d}(\mathbf p)}
\end{equation}
for all $n\ge2$, while for every $\eta>0$ there is $C_\eta>0$ such that
\begin{equation}\label{eq:main-upper}
 \cK^{\K}_{m,d,\mathbf p}(n)\le C_\eta n^{\alpha_{m,d}(\mathbf p)+\eta}
\end{equation}
for all $n\ge2$.  If $\sigma\le1/2$, or if $d=m$, the upper estimate holds with $\eta=0$.

More generally, if $N_n\in\{1,\ldots,n^m\}$ and
\[
 \frac{\log N_n}{\log n}\longrightarrow d\in[1,m],
\]
then
\[
 \lim_{n\to\infty}
 \frac{\log \sK^{\K}_{m,\mathbf p}(n,N_n)}{\log n}
 =\alpha_{m,d}(\mathbf p).
\]
\end{theorem}

In the diagonal case $p_1=\cdots=p_m=p$, the sparse multilinear
estimate gives lower bounds for polynomial coefficient embeddings.
Let $T$ be an $m$-linear form on $(\ell_p^k)^m$, and place its variables
in $m$ disjoint coordinate blocks of $\ell_p^{mk}$. The resulting
$m$-homogeneous polynomial has one distinct monomial for each nonzero
coefficient of $T$. Its coefficient $\ell_r$ norm is determined by
the cardinality of the support, while its polynomial norm is controlled
by the operator norm of $T$.

Weighted coefficient estimates arise from diagonal operators. Let
$1\le p\le s\le\infty$ and
\[
\frac1{r(p,s)}=\frac1p-\frac1s.
\]
For $u=(u_j)_{j=1}^n$, define
\[
D_u:\ell_s^n\longrightarrow\ell_p^n,
\qquad
D_u(z)=(u_jz_j)_{j=1}^n.
\]
H\"older's inequality gives
\[
\|D_u z\|_p\le\|u\|_{r(p,s)}\|z\|_s.
\]
If $P(z)=\sum_{|\alpha|=m}a_\alpha z^\alpha$ is defined on $\ell_p^n$,
then $P\circ D_u$ is defined on $\ell_s^n$ and has coefficient sequence
$(a_\alpha u^\alpha)_{|\alpha|=m}$. A Hardy--Littlewood coefficient
inequality for $P\circ D_u$ therefore gives a weighted coefficient
inequality for $P$.
For $s\ge m$, set
\[
 q(m,s):=
 \begin{cases}
 \infty,&s=m,\\[1mm]
 \dfrac{s}{s-m},&m<s<2m,\\[2mm]
 \dfrac{2ms}{ms+s-2m},&2m\le s<\infty,\\[2mm]
 \dfrac{2m}{m+1},&s=\infty.
 \end{cases}
\]

For $s>m$, these are the classical polynomial Hardy--Littlewood
exponents, with $s=\infty$ giving the Bohnenblust--Hille endpoint.
At $s=m$ we use the critical maximum-coefficient continuation
$q(m,m)=\infty$. This endpoint was studied in
\cite{CavalcanteNunez2016}; the exact complex constant needed below
will follow from the sharp coefficient estimate in
Proposition~\ref{whl:prop:coefficient-functional}.

Let $H^{\mathrm{pol}}_{m,s}(\mathbb C)$ be the optimal constant in the corresponding polynomial coefficient inequality
\begin{equation}\label{whl:eq:classical-HL}
 \left\|(a_\alpha)_{|\alpha|=m}\right\|_{\ell_{q(m,s)}}
 \le H^{\mathrm{pol}}_{m,s}(\mathbb C)\|P\|_s.
\end{equation}
If $1\le p\le s$ and
\begin{equation}\label{whl:eq:rps-intro}
 \frac1{r(p,s)}=\frac1p-\frac1s,
\end{equation}
define
\[
 \mathcal W_{m,p,s}(P):=
 \sup_{\|u\|_{r(p,s)}\le1}
 \left\|(a_\alpha u^\alpha)_{|\alpha|=m}\right\|_{\ell_{q(m,s)}}.
\]
Equivalently, $\mathcal W_{m,p,s}(P)$ is the supremum of the
$\ell_{q(m,s)}$ norm of the coefficient sequence of $P\circ D_u$
over the unit ball of $\ell_{r(p,s)}^n$.

\begin{theorem}\label{thm:B}
Let $m\ge1$, $m\le s\le\infty$, and $1\le p\le s$. Then
\begin{equation}\label{whl:eq:intro-transfer-functional}
 \mathcal W_{m,p,s}(P)
 \le H^{\mathrm{pol}}_{m,s}(\mathbb C)\|P\|_p.
\end{equation}
If $p=s\ge m$, then
\[
 \mathcal W_{m,p,p}(P)
 =\left\|(a_\alpha)_{|\alpha|=m}\right\|_{\ell_{q(m,p)}},
\]
so the new functional coincides exactly with the classical Hardy--Littlewood coefficient norm, with the same optimal constant. Moreover, if $1\le p\le m$ and
\[
 r_{m,p}:=\begin{cases}\dfrac{mp}{m-p},&p<m,\\[1mm]\infty,&p=m,\end{cases}
\]
then
\[
 \max_{|\alpha|=m}|a_\alpha u^\alpha|
 \le m\,\|u\|_{r_{m,p}}^m\|P\|_p,
\]
and the constant $m$ is optimal.
\end{theorem}

If $s=p\ge m$, then $r(p,p)=\infty$. Taking $u=(1,\ldots,1)$
shows that $\mathcal W_{m,p,p}(P)$ is at least the classical
Hardy--Littlewood coefficient norm. The reverse inequality follows
from $|u^\alpha|\le1$ whenever $\|u\|_\infty\le1$. Thus the two
functionals coincide, including their optimal constants.

If $s=m$ and $p\le m$, let
\[
\frac1{r_{m,p}}=\frac1p-\frac1m,
\]
with $r_{m,m}=\infty$. For $u\in\ell_{r_{m,p}}^n$, the diagonal operator
\[
D_u:\ell_m^n\longrightarrow\ell_p^n
\]
transfers the maximum-coefficient estimate at $s=m$ to a weighted
maximum-coefficient estimate on $\ell_p^n$.

For the coefficient/sup-norm application, let $2\le p\le m$, set
\[
 h:=\lceil p\rceil-1,\qquad c:=1-\frac hp,\qquad D:=m-h,\qquad G:=\frac{2m}{p}-1,
\]
and define
\[
 \beta_{m,p}(r):=
 \sup_{1\le d\le m}\left\{\frac dr-\alpha_{m,d}(p)\right\},
 \qquad 2\le r\le\infty.
\]
Let $\mathfrak A^m_{p,r}(n)$ be the least constant such that
\[
 \|(a_\alpha)_{|\alpha|=m}\|_{\ell_r}
 \le \mathfrak A^m_{p,r}(n)\|P\|_p
\]
for every $m$-homogeneous polynomial on $\mathbb C^n$.

\begin{theorem}\label{thm:C}
Let $m\ge2$, $2\le p\le m$, and $2\le r\le\infty$.  Then
\begin{equation}\label{eq:poly-A-limit}
 \lim_{n\to\infty}
 \frac{\log \mathfrak A_{p,r}^m(n)}{\log n}
 =\beta_{m,p}(r).
\end{equation}
Equivalently,
\[
 \mathfrak A_{p,r}^m(n)=n^{\beta_{m,p}(r)+o(1)}.
\]
The exponent admits the following explicit form.  If $D\ge G$, then
\begin{equation}\label{eq:beta-case-DgeG}
 \beta_{m,p}(r)=\frac Gr.
\end{equation}
If $D<G$, then necessarily $p>2$.  Set
\begin{equation}\label{eq:dstar-def}
 d_*:=\frac{G-2cD}{1-2c}.
\end{equation}
For $x:=1/r\in[0,1/2]$ one has
\begin{equation}\label{eq:beta-case-DltG}
 \beta_{m,p}(r)=
 \begin{cases}
 Dx, & 0\le x\le c,\\[1.5mm]
 Dc+\displaystyle\frac{x-c}{1/2-c}
 \left(\frac G2-Dc\right),
 & c\le x\le1/2.
 \end{cases}
\end{equation}
In particular, \eqref{eq:poly-A-limit} determines the optimal polynomial power in the part of the strip $2\le p\le m$, $2\le r<\infty$ which is left undetermined in the diagram of \cite[Theorem~2.1]{GalicerMansillaMuro}.
\end{theorem}

Suppose that, with a constant independent of $n$,
\[
\left\|(a_\alpha u^\alpha)_{|\alpha|=m}\right\|_{\ell_q}
\le C\|u\|_t^m\|P\|_p
\]
holds for every $m$-homogeneous polynomial $P$ on $\ell_p^n$.
Taking $u=n^{-1/t}(1,\ldots,1)$ reduces this estimate to an
unweighted coefficient inequality. Hence the growth exponent in
Theorem~\ref{thm:C} imposes a necessary condition on $t$.
Diagonal transfer of Hardy--Littlewood coefficient inequalities
provides sufficient conditions.

The sharp multiplier results can be stated in one theorem. Put
\[
 q_m:=\frac{2m}{m+1}.
\]

\begin{theorem}\label{thm:D}
Let $m\ge2$.
\begin{enumerate}[label=\textup{(\roman*)}]
\item If $2\le p\le m$ and $1\le q\le2$, define $t_*=t_*(m,p,q)$ by
\begin{equation}\label{whl:eq:tstar-intro}
 \frac1{t_*}=\frac1p+\frac1q-\frac{m+1}{2m}.
\end{equation}
Then
\begin{equation}\label{whl:eq:sharp-strip-intro}
 \left(\sum_{|\alpha|=m}|a_\alpha u^\alpha|^q\right)^{1/q}
 \le C_{m,p,q}\|u\|_{t_*}^m\|P\|_p.
\end{equation}
holds dimension-free, and $t_*$ is the largest admissible weight exponent.
\item If $p=m$ and $1\le q\le\infty$, define $\rho_m(q)$ by
\begin{equation}\label{whl:eq:rho-critical-intro}
 \frac1{\rho_m(q)}=
 \begin{cases}
 \dfrac1q-\dfrac{m-1}{2m},&1\le q\le2,\\[2mm]
 \dfrac1{mq},&2\le q\le\infty.
 \end{cases}
\end{equation}
Then $\rho_m(q)$ is the largest admissible diagonal weight exponent.
\item Let $2\le p\le m$ and $2\le q\le\infty$, and put
\[
 h:=\lceil p\rceil-1,\qquad
 D:=m-h,\qquad
 G:=\frac{2m}{p}-1.
\]
If $q=\infty$, then the largest admissible diagonal weight exponent is
\[
 t_{\rm crit}=\infty
\]
for every $2\le p\le m$.

If $2\le q<\infty$ and $D\ge G$, then the largest admissible weight
exponent satisfies
\[
 \frac1{t_{\rm crit}}
 =
 \frac{G}{mq}
 =
 \frac2q\left(\frac1p-\frac1{2m}\right).
\]
\end{enumerate}
\end{theorem}

For
\[
2<p<m,\qquad 2<q<\infty,\qquad D<G,
\]
the coefficient growth exponent in Theorem~\ref{thm:C} yields
a necessary condition for a dimension-free weighted coefficient
estimate. A sufficient condition follows from the quadratic
coefficient estimate and the maximum-coefficient estimate. Both
conditions are stated in Section~\ref{sec:remaining}.

\section{Notation and common parameters}\label{sec:notation}

All scalar-valued polynomials and multilinear forms are over $\mathbb C$ unless the scalar field $\mathbb K\in\{\mathbb R,\mathbb C\}$ is displayed. For an $m$-homogeneous polynomial
\[
 P(z)=\sum_{|\alpha|=m}a_\alpha z^\alpha
 \qquad (z\in\mathbb C^n),
\]
we write
\[
 \|P\|_p:=\sup_{\|z\|_p\le1}|P(z)|.
\]
The principal symbols are collected below; auxiliary symbols used only in one proof are defined locally.

\begin{center}
\renewcommand{\arraystretch}{1.16}
\begin{tabularx}{\textwidth}{@{}l X@{}}
\toprule
Symbol & Meaning \\
\midrule
$\sK^{\mathbb K}_{m,\mathbf p}(n,N)$ & Smallest norm of an $m$-linear form on the indicated $\ell_{p_j}^n$ spaces with at least $N$ active unimodular coefficients. \\
$\cK^{\mathbb K}_{m,d,\mathbf p}(n)$ & The preceding quantity with $N=\lfloor n^d\rfloor$. \\
$\alpha_{m,d}(\mathbf p)$ & Sparse KSZ growth exponent in Theorem~\ref{thm:A}. \\
$\mathfrak A^m_{p,r}(n)$ & Polynomial coefficient/sup-norm equivalence constant. \\
$\beta_{m,p}(r)$ & Its polynomial growth exponent in Theorem~\ref{thm:C}. \\
$q(m,s)$ & Classical polynomial Hardy--Littlewood coefficient exponent associated with $s\ge m$. \\
$\mathcal W_{m,p,s}(P)$ & Weighted Hardy--Littlewood coefficient functional in Theorem~\ref{thm:B}. \\
$r(p,s)$ & Diagonal weight exponent given by $1/r=1/p-1/s$. \\
\bottomrule
\end{tabularx}
\end{center}

\section{Proof of Theorem \ref{thm:A}: sparse Kahane--Salem--Zygmund law}\label{sec:proof-A}

\subsection{Classical coefficient estimates}

For an $m$-linear form
\[
 T:\ell_{p_1}^n\times\cdots\times\ell_{p_m}^n\longrightarrow\K
\]
write $a_{i_1,\ldots,i_m}:=T(e_{i_1},\ldots,e_{i_m})$.

For $\K=\R$, the complex coefficient estimates used below are transferred by canonical complexification; the resulting constants depend only on the fixed parameters and do not affect any power of $n$.

\begin{lemma}\label{lem:l2-critical}
Let $q_1,\ldots,q_m\in[2,\infty]$ satisfy
\[
 \sum_{j=1}^m\frac1{q_j}=\frac12.
\]
Then there is $C=C(m,q_1,\ldots,q_m)\ge1$ such that every $m$-linear form
$T:\ell_{q_1}^n\times\cdots\times\ell_{q_m}^n\to\K$ satisfies
\begin{equation}\label{eq:l2-critical}
 \left(\sum_{\mathbf i\in[n]^m}|a_{\mathbf i}|^2\right)^{1/2}
 \le C\|T\|.
\end{equation}
\end{lemma}

\begin{proof}
This is the critical anisotropic Hardy--Littlewood inequality; see \cite[Proposition~4.1]{DimantSevillaPeris}.  The dependence of the constant on the fixed exponents is harmless in all applications below.
\end{proof}

\begin{lemma}\label{lem:sectional-coeff}
Let $1\le k\le m$ and let $q_1,\ldots,q_k\in[2,\infty]$ satisfy
\[
 \frac12\le \rho:=\sum_{j=1}^k\frac1{q_j}<1.
\]
Put $1/r:=1-\rho$.  Then there is $C=C(k,q_1,\ldots,q_k)$ such that every $k$-linear form
$A:\ell_{q_1}^n\times\cdots\times\ell_{q_k}^n\to\K$ satisfies
\begin{equation}\label{eq:sectional-coeff}
 \left(\sum_{\mathbf i\in[n]^k}|A(e_{i_1},\ldots,e_{i_k})|^r\right)^{1/r}
 \le C\|A\|.
\end{equation}
\end{lemma}

\begin{proof}
This is the Hardy--Littlewood coefficient inequality in the range $1/2\le\rho<1$; see \cite[Proposition~4.1]{DimantSevillaPeris}.  The exponent $r=(1-\rho)^{-1}$ is optimal.
\end{proof}

For $p,r\ge1$, let $A^m_{p,r}(n)$ denote the least constant for which
\begin{equation}\label{eq:Apr-def}
 \left(\sum_{\mathbf i\in[n]^m}|a_{\mathbf i}|^r\right)^{1/r}
 \le A^m_{p,r}(n)\|T\|_{\cL({}^m\ell_p^n)}
\end{equation}
holds for every complex $m$-linear form $T$.

\begin{lemma}\label{lem:Bayart-sectional}
Let $1\le h\le m-1$, let $h\le p\le h+1$, and define $r$ by
\begin{equation}\label{eq:r-sectional}
 \frac1r=1-\frac hp.
\end{equation}
Then
\begin{equation}\label{eq:sectional-A}
 A^m_{p,r}(n)\lesssim n^{(m-h)/r}.
\end{equation}
The same polynomial estimate holds over $\R$.
\end{lemma}

\begin{proof}
This is \cite[Proposition~5.6]{BayartJEMS} with $k=m-h$.  For real forms, canonical complexification changes the norm by at most a constant depending only on $m$.
\end{proof}

\begin{lemma}\label{lem:HL-low-sum}
Let $p_1,\ldots,p_m\in[2,\infty]$ and
\[
 \sigma=\sum_{j=1}^m\frac1{p_j}\le\frac12.
\]
Define $q$ by
\begin{equation}\label{eq:q-HL-aniso}
 \frac1q=\frac{m+1-2\sigma}{2m}.
\end{equation}
Then every $m$-linear form $T:\ell_{p_1}^n\times\cdots\times\ell_{p_m}^n\to\K$ satisfies
\begin{equation}\label{eq:HL-low-sum}
 \left(\sum_{\mathbf i\in[n]^m}|a_{\mathbf i}|^q\right)^{1/q}
 \le C_{m,\mathbf p}\|T\|.
\end{equation}
\end{lemma}

\begin{proof}
This is the Praciano--Pereira/Hardy--Littlewood inequality in the anisotropic range $\sigma\le1/2$; see \cite{PracianoPereira,AronNunezPellegrinoSerrano}.  The exponent in \eqref{eq:q-HL-aniso} is optimal.
\end{proof}

\begin{lemma}\label{lem:classical-KSZ}
Let $m\ge2$, $p_1,\ldots,p_m\in[2,\infty]$, and $s\ge1$.  There are signs $\eps_{\mathbf i}\in\{-1,1\}$, $\mathbf i\in[s]^m$, such that
\begin{equation}\label{eq:classical-KSZ}
 \left\|\sum_{\mathbf i\in[s]^m}\eps_{\mathbf i}
 x^{(1)}_{i_1}\cdots x^{(m)}_{i_m}\right\|_{
 \cL(\ell_{p_1}^s,\ldots,\ell_{p_m}^s)}
 \le C_m s^{(m+1)/2-\sigma}.
\end{equation}
\end{lemma}

\begin{proof}
For exponents $p_j\ge2$ this is the optimal full-support anisotropic unimodular estimate; see \cite{AlbuquerqueRezende,PellegrinoSerranoSilva}.  The constant may be chosen depending only on $m$, and the same real sign form works over both scalar fields.
\end{proof}

\subsection{Lower bounds}
For the lower bounds let
\[
 T(x^{(1)},\ldots,x^{(m)})
 =\sum_{\mathbf i\in\Lambda}\eps_{\mathbf i}
 x^{(1)}_{i_1}\cdots x^{(m)}_{i_m},
 \qquad |\eps_{\mathbf i}|=1,
\]
where $\Lambda\subset[n]^m$ and $|\Lambda|\ge\lfloor n^d\rfloor$.

\begin{proposition}\label{prop:global-lower}
Assume $\sigma\ge1/2$.  Then
\begin{equation}\label{eq:global-lower}
 \|T\|\gtrsim_{m,\mathbf p}
 n^{(d+1)/2-\sigma}.
\end{equation}
Consequently $\|T\|\gtrsim n^{\max\{0,(d+1)/2-\sigma\}}$.
\end{proposition}

\begin{proof}
Choose $q_j\in[p_j,\infty]$ so that $\sum_j1/q_j=1/2$.  Such a choice exists because $\sigma\ge1/2$.  Lemma~\ref{lem:l2-critical} gives
\[
 |\Lambda|^{1/2}\le C_m\|T\|_{\cL(\ell_{q_1}^n,\ldots,\ell_{q_m}^n)}.
\]
For $x^{(j)}\in B_{\ell_{q_j}^n}$,
\[
 \|x^{(j)}\|_{p_j}\le n^{1/p_j-1/q_j}.
\]
Hence
\[
 \|T\|_{\cL(\ell_{q_1}^n,\ldots,\ell_{q_m}^n)}
 \le n^{\sigma-1/2}\|T\|_{\cL(\ell_{p_1}^n,\ldots,\ell_{p_m}^n)}.
\]
Since $|\Lambda|\gtrsim n^d$, the claim follows.  The estimate $\|T\|\ge1$ is obtained by evaluating $T$ at basis vectors corresponding to any active coefficient.
\end{proof}

\begin{proposition}\label{prop:section-lower}
Assume $\sigma\ge1/2$.  For every $I\in\mathcal I(\mathbf p)$,
\begin{equation}\label{eq:section-lower}
 \|T\|\gtrsim_{m,\mathbf p}
 n^{\pos{d-m+|I|}(1-\sigma_I)}.
\end{equation}
\end{proposition}

\begin{proof}
Fix $I\in\mathcal I(\mathbf p)$ and put $k=|I|$.  By averaging over the coordinates outside $I$, there is a choice of those $m-k$ coordinates for which the resulting $k$-dimensional section of $\Lambda$ has at least
\[
 \frac{|\Lambda|}{n^{m-k}}\gtrsim n^{d-m+k}
\]
points.  Fixing the complementary variables at the corresponding basis vectors produces a $k$-linear form $A$ on $\prod_{j\in I}\ell_{p_j}^n$ with $\|A\|\le\|T\|$ and all active coefficients of modulus one.  Since $1/2\le\sigma_I<1$, Lemma~\ref{lem:sectional-coeff} with $1/r=1-\sigma_I$ gives
\[
 n^{(d-m+k)/r}\lesssim\|A\|\le\|T\|
\]
when $d-m+k>0$.  If $d-m+k\le0$, the basis-vector estimate $\|T\|\ge1$ is stronger.  This proves \eqref{eq:section-lower}.
\end{proof}

\begin{proposition}\label{prop:high-lower}
Assume $\sigma\le1/2$.  Then
\begin{equation}\label{eq:high-lower}
 \|T\|\gtrsim_{m,\mathbf p}
 n^{d(m+1-2\sigma)/(2m)}.
\end{equation}
\end{proposition}

\begin{proof}
Let $q$ be given by \eqref{eq:q-HL-aniso}.  Since all active coefficients have modulus one,
\[
 |\Lambda|^{1/q}
 =\left(\sum_{\mathbf i\in[n]^m}|a_{\mathbf i}|^q\right)^{1/q}.
\]
Lemma~\ref{lem:HL-low-sum} and $|\Lambda|\gtrsim n^d$ give \eqref{eq:high-lower}.
\end{proof}

\subsection{Sparse supports and simultaneous signs}
Let $\Lambda\subset[n]^m$. For subsets $A_j\subset[n]$, the
restriction of $T_{\Lambda,\varepsilon}$ to
\[
\ell_{p_1}(A_1)\times\cdots\times\ell_{p_m}(A_m)
\]
has coefficient support
\[
\Lambda\cap(A_1\times\cdots\times A_m).
\]
The dimensions of the coordinate projections of this trace enter
the norm estimate for the restricted multilinear form. A dyadic
decomposition of vectors in the unit balls of the $\ell_{p_j}^n$
spaces then combines the estimates for rectangular restrictions
into an estimate for the full operator norm.

Simultaneous control of rectangular restrictions occurs in the
finite-block construction of
\cite[Sections~3--5]{BarbosaPellegrinoRaposoTeixeira2026}.
The local condition below incorporates the exponents $p_j$
individually and is adapted to the norm on
\[
\mathcal L(\ell_{p_1}^n,\ldots,\ell_{p_m}^n;\mathbb K).
\]
Assume $\sigma\ge1/2$.  For a nonempty finite set $F\subset[n]^m$ define
\[
 \pi_j(F):=\{i_j:(i_1,\ldots,i_m)\in F\},
\]
and
\begin{equation}\label{eq:local-parameters}
 M(F):=|F|,
 \qquad
 v_j(F):=|\pi_j(F)|.
\end{equation}
Put
\begin{equation}\label{eq:H-def}
 H_n(F):=1+\sum_{j=1}^m v_j(F)\log\frac{en}{v_j(F)}.
\end{equation}

\begin{definition}\label{def:lambda-sparse}
Let $\mathbf p\in[2,\infty]^m$ and $\lambda>0$.  A set $\Lambda\subset[n]^m$ is called \emph{$(\mathbf p,\lambda)$-sparse} if
\begin{equation}\label{eq:lambda-sparse}
 M(F)H_n(F)
 \le n^\lambda\prod_{j=1}^m v_j(F)^{2/p_j}
\end{equation}
for every nonempty $F\subset\Lambda$.
\end{definition}

Assume from now on that $\sigma\ge1/2$ and set
\begin{equation}\label{eq:L-aniso}
 L_{m,d}(\mathbf p):=
 \max\left\{
 0,\ d+1-2\sigma,\
 \max_{I\in\mathcal I(\mathbf p)}
 2\pos{d-m+|I|}(1-\sigma_I)
 \right\}.
\end{equation}
Thus $L_{m,d}(\mathbf p)=2\alpha_{m,d}(\mathbf p)$ in this range.

\begin{lemma}\label{lem:weighted-obstruction}
Let $\theta_1,\ldots,\theta_m\in[0,1]$, $\Theta:=\sum_j\theta_j\ge1$, and $d\in[1,m)$.  For $I\subset[m]$ put
\[
 \Theta_I:=\sum_{j\in I}\theta_j.
\]
Set
\[
 L:=\max\left\{
 0,\ d+1-\Theta,\
 \max_{\substack{\varnothing\ne I\subset[m]\\1\le\Theta_I<2}}
 \pos{d-m+|I|}(2-\Theta_I)
 \right\}.
\]
Fix a nonempty $J\subset[m]$ with $\Theta_J<2$, put $C=d-m+|J|$, and let $a\in(0,1]$ and $x_j\in[0,1]$ for $j\notin J$.  If
\begin{equation}\label{eq:weighted-psi-nonnegative}
 C+a\left(\sum_{j\notin J}x_j-1\right)\ge0,
\end{equation}
then
\begin{equation}\label{eq:weighted-objective}
 a\left(2-\Theta_J-\sum_{j\notin J}\theta_jx_j\right)\le L.
\end{equation}
\end{lemma}

\begin{proof}
Write $K=[m]\setminus J$, $R=\sum_{j\in K}x_j$, and $A=2-\Theta_J$.  If either $R\ge1$ or $C\ge1-R$, then \eqref{eq:weighted-psi-nonnegative} implies $C+R-1\ge0$, hence
\[
 |K|-R\le d-1.
\]
Since $\theta_j\le1$,
\[
 A-\sum_{j\in K}\theta_jx_j
 =2-\Theta+\sum_{j\in K}\theta_j(1-x_j)
 \le d+1-\Theta\le L,
\]
and \eqref{eq:weighted-objective} follows from $a\le1$.

It remains to consider $R<1$ and $0<C<1-R$.  Then \eqref{eq:weighted-psi-nonnegative} gives
\[
 a\le\frac{C}{1-R}.
\]
If $A-\sum_{j\in K}\theta_jx_j\le0$, then \eqref{eq:weighted-objective} is immediate.  Hence we may assume that this quantity is positive.
Let
\[
 \Delta:=\left\{x=(x_j)_{j\in K}:x_j\ge0,\ \sum_{j\in K}x_j\le1-C\right\}.
\]
Its vertices are $v_0=0$ and $v_k=(1-C)e_k$, $k\in K$.  Set
\[
 N(x):=C\left(A-\sum_{j\in K}\theta_jx_j\right),
 \qquad D(x):=1-\sum_{j\in K}x_j.
\]
On $\Delta$ one has $D(x)\ge C>0$.  If $x=\sum_i\lambda_i v_i$ is a convex decomposition into the vertices of $\Delta$, then $N$ and $D$ are affine and
\[
 \frac{N(x)}{D(x)}
 =\sum_i
 \frac{\lambda_iD(v_i)}{D(x)}\,
 \frac{N(v_i)}{D(v_i)}.
\]
The coefficients $\lambda_iD(v_i)/D(x)$ are nonnegative and sum to $1$.  Hence $N(x)/D(x)$ is bounded by the largest of its vertex values.
At $x=0$ the value is $CA$.  If $\Theta_J\ge1$, this is one of the sectional quantities in the definition of $L$; if $\Theta_J<1$, then
\[
 (d+1-\Theta)-CA
 =(1-C)(A-1)+\sum_{j\in K}(1-\theta_j)\ge0.
\]
At the remaining vertices, $x=(1-C)e_k$ for $k\in K$, the value is $A-(1-C)\theta_k$, and
\[
 d+1-\Theta-\bigl(A-(1-C)\theta_k\bigr)
 =C(1-\theta_k)+\sum_{j\in K\setminus\{k\}}(1-\theta_j)\ge0.
\]
Thus every vertex value is at most $L$, proving \eqref{eq:weighted-objective}.
\end{proof}

\begin{proposition}\label{prop:quant-sparse}
Let $m\ge2$, $1\le d<m$, $\mathbf p\in[2,\infty]^m$, and $\sigma\ge1/2$.  If
\begin{equation}\label{eq:lambda-above}
 \lambda>L_{m,d}(\mathbf p),
\end{equation}
then there are constants $c>0$ and $n_0\in\N$ such that for every $n\ge n_0$ there exists a $(\mathbf p,\lambda)$-sparse set $\Lambda_n\subset[n]^m$ satisfying
\begin{equation}\label{eq:support-size-quant}
 n^d\le|\Lambda_n|\le3n^d.
\end{equation}
\end{proposition}

\begin{proof}
Put $\theta_j:=2/p_j\in[0,1]$ and $\Theta=\sum_j\theta_j=2\sigma\ge1$, and set
\[
 \varepsilon_0:=\frac{\lambda-L_{m,d}(\mathbf p)}{10}>0.
\]
All perturbations below are chosen smaller than $\varepsilon_0$.

For each sufficiently large $n$, let
\[
 \Omega_n:=\{0,1\}^{[n]^m}
\]
with the product probability measure $\mu_n$ for which each coordinate has mass $q:=2n^{d-m}$ at $1$. Increasing $n_0$ if necessary, we may and shall assume that
$q=2n^{d-m}\le1$. For $\omega\in\Omega_n$ write
\[
 \Lambda(\omega):=\{\mathbf i\in[n]^m:\omega_{\mathbf i}=1\}.
\]
Thus $\int_{\Omega_n}|\Lambda(\omega)|\,d\mu_n(\omega)=2n^d$.  We shall exhibit a measurable set of positive $\mu_n$-measure on which the support has the required cardinality and sparsity.

Let $B_n$ be the set of $\omega$ for which $\Lambda(\omega)$ is not $(\mathbf p,\lambda)$-sparse.  For integers $M,v_1,\ldots,v_m$, let $B_n(M,\mathbf v)$ be the part of $B_n$ witnessed by a nonempty set $F\subset\Lambda(\omega)$ with
\[
 M=|F|,\qquad v_j=|\pi_j(F)|,
 \qquad P=\prod_{j=1}^m v_j,\qquad V=\max_jv_j.
\]
The union of these measurable sets is $B_n$.  Counting the possible projections and then the possible $M$-point subsets of their Cartesian product gives
\begin{equation}\label{eq:first-moment-master-new}
 \mu_n\bigl(B_n(M,\mathbf v)\bigr)
 \le
 \exp\left(\sum_{j=1}^m v_j\log\frac{en}{v_j}\right)
 \left(\frac{eqP}{M}\right)^M.
\end{equation}
Indeed, this is simply the integral over $\Omega_n$ of the number of witnesses in the class.  Also $H_n(F)\le C_mV\log n$.

\medskip
\noindent\emph{Regime I: $M\ge Vn^\tau$.}
If $F$ violates \eqref{eq:lambda-sparse}, then, after absorbing $\log n$ into $n^\eta$,
\[
 M\ge n^{\lambda-\eta}\frac{\prod_jv_j^{\theta_j}}{V}.
\]
Since $\theta_j-1\le0$ and $v_j\le V\le n$,
\[
 \frac{M}{qP}
 \ge \frac12 n^{\lambda-\eta+m-d}
 \frac{\prod_jv_j^{\theta_j-1}}{V}
 \ge \frac12 n^{\lambda-\eta+\Theta-d-1}.
\]
For the last inequality, $v_j^{\theta_j-1}\ge n^{\theta_j-1}$ and $V^{-1}\ge n^{-1}$.  Moreover,
\[
 L_{m,d}(\mathbf p)\ge d+1-\Theta,
\]
so \eqref{eq:lambda-above} gives $\lambda+\Theta-d-1>0$.  We may therefore choose the logarithmic loss $\eta>0$ so small that
\[
 \kappa:=\lambda-\eta+\Theta-d-1>0.
\]
Then $M/(qP)\ge\frac12 n^\kappa$, and for all sufficiently large $n$,
\[
 \left(\frac{eqP}{M}\right)^M
 \le \exp\!\left(-\frac{\kappa}{2}M\log n\right).
\]
The projection-counting exponent also has a uniform estimate:
\[
 \sum_{j=1}^m v_j\log\frac{en}{v_j}
 \le C_mV\log n
 \le C_m n^{-\tau}M\log n,
\]
because $M\ge Vn^\tau$.  Thus this contribution is $o(M\log n)$ uniformly in Regime~I.  Consequently
\begin{equation}\label{eq:measure-regimeI}
 \mu_n\bigl(B_n(M,\mathbf v)\bigr)\le n^{-cM}
\end{equation}
throughout Regime~I.

\medskip
\noindent\emph{Regime II: $V\le M<Vn^\tau$.}
Write $V=n^a$, so $a\in(0,1]$, with $\tau,\eta>0$ to be chosen below, and put
\[
 b_j:=\frac{\log v_j}{\log n}\in[0,a].
\]
If $F$ violates \eqref{eq:lambda-sparse}, then
\[
 n^\lambda\prod_{j=1}^m v_j^{\theta_j}<M H_n(F).
\]
Since $H_n(F)\le C_mV\log n$ and $M<Vn^\tau$, we obtain
\[
 n^\lambda\prod_{j=1}^m v_j^{\theta_j}
 < C_mV^2n^\tau\log n.
\]
Choose $n$ so large that $\log_n(C_m\log n)\le\eta$.  Taking logarithms to base $n$ gives
\begin{equation}\label{eq:weighted-violation-regimeII}
 \lambda+\sum_{j=1}^m\theta_j b_j
 \le2a+\tau+\eta.
\end{equation}
In particular,
\[
 a\ge\frac{\lambda-\tau-\eta}{2}.
\]
Let
\[
 J:=\{j:b_j\ge a-\tau\},\qquad r:=|J|.
\]
The set $J$ is nonempty because $V=\max_jv_j$, so $b_j=a$ for at least one $j$.  Writing
\[
 \Theta_J:=\sum_{j\in J}\theta_j,
\]
and using $b_j\ge a-\tau$ on $J$, \eqref{eq:weighted-violation-regimeII} yields
\begin{equation}\label{eq:regimeII-J-rough}
 \lambda\le a(2-\Theta_J)+(\Theta_J+1)\tau+\eta.
\end{equation}
After $\tau,\eta$ are chosen so that $(m+1)\tau+\eta<\lambda$, \eqref{eq:regimeII-J-rough} forces $\Theta_J<2$.

For $j\notin J$ put $x_j:=b_j/a\in[0,1]$.  We next make the first-moment exponent completely explicit.  Set
\[
 t:=\frac{\log M}{\log n};
 \qquad a\le t<a+\tau.
\]
From \eqref{eq:first-moment-master-new} and $q=2n^{d-m}$,
\begin{align*}
 \frac{1}{M\log n}\log\mu_n\bigl(B_n(M,\mathbf v)\bigr)
 &\le
 \sum_{j=1}^m\frac{v_j}{M}
 \frac{\log(en/v_j)}{\log n}
 +\frac{\log(eqP/M)}{\log n}\\
 &=d-m+\sum_{j=1}^m b_j-t
   +\sum_{j=1}^m\frac{v_j}{M}(1-b_j)
   +E_n,
\end{align*}
where
\[
 0\le E_n
 \le\frac{m+\log(2e)}{\log n}.
\]
Indeed, $\log(en/v_j)=1+(1-b_j)\log n$, $\sum_jv_j/M\le m$, and the constant $2e$ comes from $eq=2e\,n^{d-m}$.

For $j\in J$, the inequalities $v_j\le M$ and $0\le b_j\le1$ give
\[
 b_j+\frac{v_j}{M}(1-b_j)\le1.
\]
For $j\notin J$, since $b_j<a-\tau$ and $t\ge a$,
\[
 \frac{v_j}{M}=n^{b_j-t}\le n^{-\tau}.
\]
Using also $t\ge a$, we obtain the uniform estimate
\begin{equation}\label{eq:entropy-regimeII-aniso}
 \frac{1}{M\log n}\log\mu_n\bigl(B_n(M,\mathbf v)\bigr)
 \le d-m+r+a\left(\sum_{j\notin J}x_j-1\right)
 +m n^{-\tau}+\frac{m+\log(2e)}{\log n}.
\end{equation}

It remains to obtain a uniform negative gap in the first term on the right.  Put
\[
 a_0:=\min\left\{\frac\lambda4,1\right\},
\]
and, for each nonempty $J\subset[m]$ with $\Theta_J<2$, define on the compact set
\[
 K_J:=[a_0,1]\times[0,1]^{m-r}
\]
the continuous functions
\[
 \Phi_J(a,x):=a\left(2-\Theta_J-\sum_{j\notin J}\theta_jx_j\right),
 \qquad
 \Psi_J(a,x):=d-m+r+a\left(\sum_{j\notin J}x_j-1\right).
\]
Let
\[
 S_J:=\{(a,x)\in K_J:\Phi_J(a,x)\ge L_{m,d}(\mathbf p)+3\varepsilon_0\}.
\]
If $S_J$ is nonempty, then $\Psi_J<0$ everywhere on $S_J$.  Otherwise a point of $S_J$ with $\Psi_J\ge0$ would satisfy the hypothesis of Lemma~\ref{lem:weighted-obstruction}, which would imply $\Phi_J\le L_{m,d}(\mathbf p)$, a contradiction.  Compactness therefore gives
\[
 \max_{S_J}\Psi_J<0.
\]
There are only finitely many possible $J$.  Omitting the empty sets $S_J$, we can choose $\delta>0$, depending only on $m,d,\mathbf p,\lambda$, such that
\begin{equation}\label{eq:weighted-stability}
 \Phi_J(a,x)\ge L_{m,d}(\mathbf p)+3\varepsilon_0
 \quad\Longrightarrow\quad
 \Psi_J(a,x)\le-\delta
\end{equation}
for every admissible $J$ and every $(a,x)\in K_J$.

Now choose $\tau,\eta>0$ so small that
\[
 (m+1)\tau+\eta<\varepsilon_0,
 \qquad
 \tau+\eta<\frac\lambda2.
\]
Then every violating witness in Regime~II satisfies
\[
 a\ge\frac{\lambda-\tau-\eta}{2}\ge\frac\lambda4.
\]
If $\lambda>4$, this is impossible because $a\le1$, so Regime~II is empty.  Otherwise $a\ge a_0$, and the compactness estimate above applies.  Moreover, from \eqref{eq:weighted-violation-regimeII}, the inequalities $b_j\ge a-\tau$ for $j\in J$, and $b_j=ax_j$ for $j\notin J$, we get
\begin{align*}
 \lambda
 &\le
 a\left(2-\Theta_J-\sum_{j\notin J}\theta_jx_j\right)
 +(\Theta_J+1)\tau+\eta\\
 &\le \Phi_J(a,x)+(m+1)\tau+\eta.
\end{align*}
Thus the defining inequality of a violating witness places $(a,x)$ inside the compact stability region: the error term in the preceding display is smaller than $\varepsilon_0$. Since $\lambda=L_{m,d}(\mathbf p)+10\varepsilon_0$, it follows that
\[
 \Phi_J(a,x)\ge L_{m,d}(\mathbf p)+9\varepsilon_0,
\]
so \eqref{eq:weighted-stability} gives $\Psi_J(a,x)\le-\delta$.  Finally, with the now fixed $\tau>0$, choose $n$ sufficiently large that
\[
 m n^{-\tau}+\frac{m+\log(2e)}{\log n}\le\frac\delta2.
\]
Substitution into \eqref{eq:entropy-regimeII-aniso} yields, uniformly over Regime~II,
\[
 \frac{1}{M\log n}\log\mu_n\bigl(B_n(M,\mathbf v)\bigr)
 \le-\frac\delta2.
\]
Thus
\begin{equation}\label{eq:measure-regimeII}
 \mu_n\bigl(B_n(M,\mathbf v)\bigr)\le n^{-cM}.
\end{equation}

There are only polynomially many parameter vectors $(M,v_1,\ldots,v_m)$.  In Regime~I, $M\ge Vn^\tau\ge n^\tau$; in Regime~II, the lower bound on $a$ again makes $M$ a positive power of $n$.  The exceptional case $V=1$ has $M=1$ and cannot violate the sparsity condition for large $n$.  Summing \eqref{eq:measure-regimeI} and \eqref{eq:measure-regimeII} therefore gives
\[
 \mu_n(B_n)=o(1).
\]

Finally let
\[
 C_n:=\{\omega:n^d\le|\Lambda(\omega)|\le3n^d\}.
\]
The standard binomial tail estimate gives
\[
 \mu_n(C_n^c)\le2e^{-c'n^d}=o(1).
\]
Hence
\[
 \mu_n(C_n\setminus B_n)
 \ge1-\mu_n(B_n)-\mu_n(C_n^c)>0
\]
for all sufficiently large $n$.  Any $\omega$ in this positive-measure set provides the required deterministic support $\Lambda_n=\Lambda(\omega)$.
\end{proof}

In the first-moment estimate, Regime I gives the global threshold $d+1-2\sigma$.  Regime II reduces every obstruction to that threshold or to a set $I$ with $1/2\le\sigma_I<1$, giving $2\pos{d-m+|I|}(1-\sigma_I)$.

For $A_1,\ldots,A_m\subset[n]$ and $\Lambda\subset[n]^m$, write
\[
 F_A:=\Lambda\cap(A_1\times\cdots\times A_m).
\]
If $F_A\ne\varnothing$, replacing $A_j$ by $\pi_j(F_A)$ does not change $F_A$.  We call this the effective representation of the trace.

A single assignment of signs gives a simultaneous norm estimate for
the restrictions to all rectangular coordinate subspaces; see
\cite[Proposition~4.1]{BarbosaPellegrinoRaposoTeixeira2026}.

\begin{proposition}\label{prop:rect-sign}
For every $m\ge2$ there exists $C_m>0$ such that for every $n$ and every $\Lambda\subset[n]^m$ one can choose signs $\eps_{\mathbf i}\in\{-1,1\}$, $\mathbf i\in\Lambda$, so that every nonempty rectangular trace $F$ satisfies
\begin{equation}\label{eq:rect-sign}
 \sup_{\substack{|z_i^{(j)}|\le1\\
 i\in\pi_j(F),\ 1\le j\le m}}
 \left|
 \sum_{\mathbf i\in F}\eps_{\mathbf i}
 z^{(1)}_{i_1}\cdots z^{(m)}_{i_m}
 \right|
 \le C_m\sqrt{M(F)H_n(F)}.
\end{equation}
\end{proposition}

\begin{proof}
Let $\Omega_\Lambda:=\{-1,1\}^{\Lambda}$ with its normalized counting measure $\nu_\Lambda$.  Thus a point $\omega\in\Omega_\Lambda$ is a complete sign assignment $\eps_{\mathbf i}(\omega)$, and the coordinate functions are independent Rademacher functions.  We show that the set of assignments satisfying \eqref{eq:rect-sign} for every trace has positive $\nu_\Lambda$-measure.  Fix an effective trace
\[
 F=\Lambda\cap(B_1\times\cdots\times B_m),
 \qquad B_j=\pi_j(F),
\]
and write $v_j=|B_j|$ and $M=|F|$.  For fixed $z^{(j)}$ with $|z_i^{(j)}|\le1$, Hoeffding's inequality applied to the real and imaginary parts gives
\begin{equation}\label{eq:hoeffding-sign}
 \nu_\Lambda\left\{
 \left|\sum_{\mathbf i\in F}\eps_{\mathbf i}
 z^{(1)}_{i_1}\cdots z^{(m)}_{i_m}\right|>t
 \right\}
 \le4\exp\left(-c\frac{t^2}{M}\right).
\end{equation}
By the maximum modulus principle, the supremum in \eqref{eq:rect-sign} may be taken on
\[
 \T^{B_1}\times\cdots\times\T^{B_m}.
\]
Fix $\delta=(4m)^{-1}$ and choose a $\delta$-net $\mathcal N_j$ of $\T^{B_j}$ in the sup norm with
\[
 |\mathcal N_j|\le C_m^{v_j}.
\]
Let $K_F$ denote the supremum over
$\T^{B_1}\times\cdots\times\T^{B_m}$. For
\[
 z=(z^{(1)},\ldots,z^{(m)})
 \in\T^{B_1}\times\cdots\times\T^{B_m},
\]
choose $w^{(j)}\in\mathcal N_j$ such that
\[
 \|z^{(j)}-w^{(j)}\|_\infty\le\delta.
\]
Writing
\[
 S_F(z^{(1)},\ldots,z^{(m)})
 :=
 \sum_{\mathbf i\in F}\eps_{\mathbf i}
 z^{(1)}_{i_1}\cdots z^{(m)}_{i_m},
\]
multilinearity and telescoping one coordinate group at a time give
\[
 \left|
 S_F(z^{(1)},\ldots,z^{(m)})
 -
 S_F(w^{(1)},\ldots,w^{(m)})
 \right|
 \le m\delta K_F.
\]
Indeed, after normalizing each difference
$z^{(j)}-w^{(j)}$ by its $\ell_\infty$ norm, every intermediate
vector lies in the unit polydisc, whose supremum agrees with the torus
supremum by the maximum modulus principle. Hence
\[
 K_F
 \le
 \max_{(w^{(1)},\ldots,w^{(m)})\in
 \mathcal N_1\times\cdots\times\mathcal N_m}
 |S_F(w^{(1)},\ldots,w^{(m)})|
 +m\delta K_F.
\]
Since $\delta=(4m)^{-1}$,
\[
 K_F\le\frac43
 \max_{\mathcal N_1\times\cdots\times\mathcal N_m}|S_F|
 \le
 2\max_{\mathcal N_1\times\cdots\times\mathcal N_m}|S_F|.
\]

For fixed projection sizes, the number of effective rectangles is at most
\[
 \prod_{j=1}^m\binom{n}{v_j}
 \le\exp\left(\sum_{j=1}^m v_j\log\frac{en}{v_j}\right).
\]
The total number of net points over one rectangle is at most
\[
 C_m^{v_1+\cdots+v_m}.
\]
Since $v_j\le v_j\log(en/v_j)$, both costs are bounded by $\exp(C_m'H_n(F))$.  Taking
\[
 t=K_m\sqrt{M H_n(F)}
\]
in \eqref{eq:hoeffding-sign}, with $K_m$ sufficiently large, gives a failure probability at most
\[
 e^{-A_mH_n(F)}
\]
for all effective rectangles with the fixed projection sizes, where $A_m$ may be chosen arbitrarily large by increasing $K_m$. For each effective rectangle and each point of the corresponding product net, let $E(F,z)$ denote the exceptional set. After the preceding union over rectangles and net points, the remaining failure probability for a fixed size vector $\mathbf v=(v_1,\ldots,v_m)$ is bounded by
\[
 \exp\!\left[-A_m\left(1+\sum_{j=1}^m v_j\log\frac{en}{v_j}\right)\right].
\]
Since there are at most $n^m$ size vectors and, for $1\le v_j\le n$, one has $v_j\log(en/v_j)\ge\log(en)$, the total exceptional measure is at most
\[
 n^m\exp\!\bigl(-A_m(1+m\log(en))\bigr)<1
\]
for $A_m$ sufficiently large. Its complement therefore has positive $\nu_\Lambda$-measure.  Every sign assignment in that complement satisfies \eqref{eq:rect-sign} for all traces simultaneously.
\end{proof}

The estimate on
$\mathcal L(\ell_{p_1}^n,\ldots,\ell_{p_m}^n;\mathbb K)$
follows by applying the rectangular estimate to dyadic coordinate
blocks of vectors in the unit balls of the $\ell_{p_j}^n$ spaces.
For a trace $F$, the anisotropic coordinate contribution is
\[
\prod_{j=1}^m v_j(F)^{2/p_j};
\]
compare \cite[Proposition~5.1]{BarbosaPellegrinoRaposoTeixeira2026}.

\begin{proposition}\label{prop:local-to-norm}
Let $m\ge2$, $\mathbf p\in[2,\infty]^m$, $\lambda>0$, and let $\Lambda\subset[n]^m$ be $(\mathbf p,\lambda)$-sparse.  Then there are signs $\eps_{\mathbf i}\in\{-1,1\}$ such that
\begin{equation}\label{eq:local-to-norm}
 \left\|\sum_{\mathbf i\in\Lambda}\eps_{\mathbf i}
 x^{(1)}_{i_1}\cdots x^{(m)}_{i_m}\right\|_{
 \cL(\ell_{p_1}^n,\ldots,\ell_{p_m}^n)}
 \le C_{m,\mathbf p}n^{\lambda/2}(1+\log n)^m.
\end{equation}
\end{proposition}

\begin{proof}
Choose the signs from Proposition~\ref{prop:rect-sign}.  Fix $x^{(j)}\in B_{\ell_{p_j}^n}$.  For each $j$, rearrange the coordinates of $x^{(j)}$ in nonincreasing order of modulus and divide the indices into dyadic rank blocks
\[
 A_{j,0},A_{j,1},\ldots,A_{j,L_j},
 \qquad L_j\le\lceil\log_2n\rceil.
\]
For $i\in A_{j,k}$,
\begin{equation}\label{eq:rank-bound-new}
 |x_i^{(j)}|\le C_{p_j}|A_{j,k}|^{-1/p_j}.
\end{equation}
For a block tuple $\mathbf k=(k_1,\ldots,k_m)$ set
\[
 F_{\mathbf k}:=\Lambda\cap
 (A_{1,k_1}\times\cdots\times A_{m,k_m}).
\]
If this trace is nonempty, let $v_j=|\pi_j(F_{\mathbf k})|$ and $M=|F_{\mathbf k}|$.  Since $v_j\le|A_{j,k_j}|$ and $p_j\ge2$, we have
\[
 |A_{j,k_j}|^{-1/p_j}\le v_j^{-1/p_j}.
\]
Thus replacing the block-size factors in \eqref{eq:rank-bound-new} by the larger projection-size factors preserves the upper estimate.  Proposition~\ref{prop:rect-sign} therefore gives
\begin{align*}
 \left|\sum_{\mathbf i\in F_{\mathbf k}}\eps_{\mathbf i}
 x^{(1)}_{i_1}\cdots x^{(m)}_{i_m}\right|
 &\le C_{m,\mathbf p}
 \left(\prod_{j=1}^m v_j^{-1/p_j}\right)
 \sqrt{M H_n(F_{\mathbf k})}\\
 &\le C_{m,\mathbf p}n^{\lambda/2}
\end{align*}
by \eqref{eq:lambda-sparse}.  There are at most $(1+\lceil\log_2n\rceil)^m$ block tuples, and summing over them proves \eqref{eq:local-to-norm}.
\end{proof}

\begin{corollary}\label{cor:sparse-upper}
Let $m\ge2$, $1\le d<m$, $\mathbf p\in[2,\infty]^m$, and $\sigma\ge1/2$.  For every $\eta>0$ and all sufficiently large $n$ there is a unimodular form with at least $n^d$ active coefficients and
\begin{equation}\label{eq:sparse-upper-alpha}
 \|T_n\|\lesssim_{m,d,\mathbf p,\eta}
 n^{\alpha_{m,d}(\mathbf p)+\eta}.
\end{equation}
\end{corollary}

\begin{proof}
Choose $\lambda>L_{m,d}(\mathbf p)=2\alpha_{m,d}(\mathbf p)$ so close to $2\alpha_{m,d}(\mathbf p)$ that $\lambda/2<\alpha_{m,d}(\mathbf p)+\eta/2$.  Proposition~\ref{prop:quant-sparse} and Proposition~\ref{prop:local-to-norm} give
\[
 \|T_n\|\lesssim n^{\lambda/2}(1+\log n)^m,
\]
and the logarithmic factor is at most $n^{\eta/2}$ for large $n$.
\end{proof}

\subsection{Upper constructions}
When $\sigma\le1/2$, the optimal polynomial geometry is obtained by packing the active coefficients into a smaller full box.

\begin{proposition}\label{prop:compressed-upper}
Let $m\ge2$, $d\in[1,m]$, and $\mathbf p\in[2,\infty]^m$ satisfy $\sigma\le1/2$.  Then
\begin{equation}\label{eq:compressed-upper}
 \cK^{\K}_{m,d,\mathbf p}(n)
 \lesssim_m n^{d(m+1-2\sigma)/(2m)}.
\end{equation}
\end{proposition}

\begin{proof}
Let $s=\lceil n^{d/m}\rceil$.  By Lemma~\ref{lem:classical-KSZ}, there is a full sign form on
$\ell_{p_1}^s\times\cdots\times\ell_{p_m}^s$ with $s^m\ge n^d$ active coefficients and norm at most
\[
 C_m s^{(m+1)/2-\sigma}
 \lesssim_m n^{d(m+1-2\sigma)/(2m)}.
\]
Extending the form by zero to the ambient dimension $n$ proves the claim.
\end{proof}

\begin{corollary}\label{cor:pinfty-constant}
For every fixed $m\ge2$ and $d\in[1,m]$,
\begin{equation}\label{eq:pinfty-limsup}
 \limsup_{n\to\infty}
 \frac{\cK^{\K}_{m,d,\infty}(n)}{n^{d(m+1)/(2m)}}\le1.
\end{equation}
\end{corollary}

\begin{proof}
Set $s=\lceil n^{d/m}\rceil$. By \cite{PellegrinoRaposo}, for fixed $m$ and every $\varepsilon>0$, once $s$ is large enough one may choose a full $m$-linear sign form on $(\ell_\infty^s)^m$ with norm at most
\[
 (1+\varepsilon)s^{(m+1)/2}.
\]
Extend it by zero to $(\ell_\infty^n)^m$.  Since $s^m\ge n^d$ and $s/n^{d/m}\to1$, division by $n^{d(m+1)/(2m)}$ and then $\varepsilon\downarrow0$ give \eqref{eq:pinfty-limsup}.
\end{proof}

\subsection{Completion of the proof}

\begin{proof}
Fix $d\in[1,m]$ first.  All upper constructions below use real signs and are estimated over $\C$; restriction to real vectors gives the same upper bounds over $\R$.  Suppose $\sigma\le1/2$.  Proposition~\ref{prop:high-lower} gives
\[
 \cK^{\K}_{m,d,\mathbf p}(n)
 \gtrsim n^{d(m+1-2\sigma)/(2m)},
\]
while Proposition~\ref{prop:compressed-upper} gives the reverse estimate up to a constant.  This proves \eqref{eq:main-asymptotic} in the low-sum range.

Now suppose $\sigma\ge1/2$.  Proposition~\ref{prop:global-lower}, Proposition~\ref{prop:section-lower}, and the basis-vector estimate give
\[
 \cK^{\K}_{m,d,\mathbf p}(n)
 \gtrsim n^{\alpha_{m,d}(\mathbf p)}.
\]
If $d<m$, Corollary~\ref{cor:sparse-upper} gives the matching upper power up to $n^\eta$.  If $d=m$, Lemma~\ref{lem:classical-KSZ} gives the upper power
\[
 \frac{m+1}{2}-\sigma.
\]
This is exactly the global term in \eqref{eq:alpha-aniso}.  It also dominates every sectional term.  Indeed, for $I\in\mathcal I(\mathbf p)$, writing $r=|I|$ and $J=[m]\setminus I$, the bounds $\sigma_I\ge1/2$ and $\sigma_J\le(m-r)/2$ give
\begin{align*}
 \frac{m+1}{2}-\sigma-r(1-\sigma_I)
 &=\frac{m+1-2r}{2}+(r-1)\sigma_I-\sigma_J\\
 &\ge0.
\end{align*}
Thus the full-support anisotropic exponent is recovered exactly.

For the coefficient-count formulation, first suppose $1<d<m$.  If $N_n=n^{d+o(1)}$, then for every sufficiently small $\delta>0$,
\[
 n^{d-\delta}\le N_n\le n^{d+\delta}
\]
for all large $n$.  Monotonicity of $N\mapsto\sK^{\K}_{m,\mathbf p}(n,N)$ and the fixed-$d$ result at $d-\delta$ and $d+\delta$ give the same logarithmic exponent after letting $\delta\downarrow0$.  The endpoint $d=m$ follows similarly using the full-support estimate.

At $d=1$ one must also allow sequences with $N_n<n$.  No truncation argument is needed.  If $\sigma\ge1$, evaluation on a nonzero coefficient gives the lower bound $1$, while a diagonal form with $N_n$ coefficients has norm $1$ when $N_n\le n$.  If $1/2\le\sigma<1$, the Hardy--Littlewood coefficient estimate applied directly to the whole active set gives
\[
 \sK^{\K}_{m,\mathbf p}(n,N_n)\gtrsim N_n^{\,1-\sigma},
\]
and the diagonal form gives the reverse estimate when $N_n\le n$.  If $\sigma\le1/2$, the low-sum coefficient estimate and the compressed-box construction give
\[
 \sK^{\K}_{m,\mathbf p}(n,N_n)
 =N_n^{(m+1-2\sigma)/(2m)+o(1)}.
\]
If $\sigma\ge1/2$ and $N_n>n$, fix $0<\delta<m-1$.  Since $N_n=n^{1+o(1)}$, eventually $N_n\le\lfloor n^{1+\delta}\rfloor$, and monotonicity together with the fixed-$d$ estimate already established above gives 
\[
 \sK^{\K}_{m,\mathbf p}(n,N_n)
 \le \cK^{\K}_{m,1+\delta,\mathbf p}(n)
 \le C_{\delta,\eta}n^{\alpha_{m,1+\delta}(\mathbf p)+\eta}
\]
for every $\eta>0$.  Letting $\eta,\delta\downarrow0$ and using the continuity of $d\mapsto\alpha_{m,d}(\mathbf p)$ gives the required upper exponent.  The lower exponent follows from the preceding coefficient estimate when $\sigma<1$ and from basis-vector evaluation when $\sigma\ge1$.  Since $N_n=n^{1+o(1)}$, these estimates yield precisely the exponent $\alpha_{m,1}(\mathbf p)$.  This proves the extension to $N_n=n^{d+o(1)}$ also at the endpoint $d=1$.
\end{proof}

\section{Proof of Theorem \ref{thm:B}: weighted Hardy--Littlewood transfer}\label{sec:proof-B}

\subsection{Monomial optimization and coefficient extraction}

For $\alpha\in\Nzero^n$ with $|\alpha|=m$, put
\[
 \alpha^\alpha:=\prod_{j:\,\alpha_j>0}\alpha_j^{\alpha_j}.
\]

\begin{lemma}[Monomial optimization]\label{whl:lem:monomial}
Let $1\le p<\infty$ and $|\alpha|=m$. Then
\begin{equation}\label{whl:eq:monomial-norm}
 \sup_{\|z\|_p\le1}|z^\alpha|
 =\left(\frac{\alpha^\alpha}{m^m}\right)^{1/p}.
\end{equation}
For $p=\infty$, the supremum equals $1$.
\end{lemma}

\begin{proof}
Set $x_j=|z_j|^p$. Under the constraint $\sum_jx_j\le1$,
\[
 |z^\alpha|^p=\prod_{j:\,\alpha_j>0}x_j^{\alpha_j}.
\]
Weighted arithmetic--geometric mean gives
\[
 \prod_{j:\,\alpha_j>0}x_j^{\alpha_j}
 \le \prod_{j:\,\alpha_j>0}
 \left(\frac{\alpha_j}{m}\right)^{\alpha_j},
\]
and equality holds for $x_j=\alpha_j/m$ on $\supp\alpha$. This proves \eqref{whl:eq:monomial-norm}. The case $p=\infty$ is immediate.
\end{proof}

For $|\alpha|=m$, consider the coefficient functional
\[
\delta_\alpha:
\mathcal P({}^m\ell_p^n)\longrightarrow\mathbb C,
\qquad
\delta_\alpha(P)=a_\alpha.
\]
Its norm is estimated by optimizing the radii in Cauchy's formula.
A related coefficient estimate appears in
\cite[Lemma~3.3]{BayartDefantSchlueters2019}.
\begin{proposition}[Exact coefficient functional]\label{whl:prop:coefficient-functional}
Let $1\le p\le\infty$. For every $m$-homogeneous polynomial
\[
 P(z)=\sum_{|\alpha|=m}a_\alpha z^\alpha
\]
on $\C^n$ and every $|\alpha|=m$,
\begin{equation}\label{whl:eq:exact-coeff-functional}
 |a_\alpha|
 \le
 \left(\frac{m^m}{\alpha^\alpha}\right)^{1/p}\|P\|_p,
\end{equation}
where the factor is interpreted as $1$ when $p=\infty$. The constant in \eqref{whl:eq:exact-coeff-functional} is optimal for each fixed $\alpha$.
\end{proposition}

\begin{proof}
Assume first that $p<\infty$. Choose radii $\rho_j>0$ such that
\[
 \sum_{j=1}^n \rho_j^p\le1.
\]
The corresponding polydisc
\[
 \{z\in\C^n:|z_j|\le\rho_j,\ 1\le j\le n\}
\]
is contained in $B_{\ell_p^n}$. Hence Cauchy's coefficient formula
gives
\[
 |a_\alpha|\rho^\alpha\le\|P\|_p.
\]
Optimizing over the coordinates in $\supp\alpha$ and using
\cref{whl:lem:monomial} gives
\[
 \sup_{\sum_j\rho_j^p\le1}\rho^\alpha
 =
 \left(\frac{\alpha^\alpha}{m^m}\right)^{1/p}.
\]
If $\alpha_j=0$ for some $j$, the same value is obtained by first
optimizing on $\supp\alpha$ and then letting the remaining positive
radii tend to zero. Therefore
\[
 |a_\alpha|
 \le
 \left(\frac{m^m}{\alpha^\alpha}\right)^{1/p}
 \|P\|_p.
\]
For $p=\infty$, the estimate follows directly from Cauchy's formula on
the unit polydisc. Finally, equality is attained, after normalization,
by the monomial $P(z)=z^\alpha$, so the constant is optimal.
\end{proof}

\subsection{Diagonal transfer and exact recovery of the classical theory}

Let $m\le s\le\infty$, $1\le p\le s$, and let $r$ satisfy
\[
 \frac1r=\frac1p-\frac1s.
\]
Consider the diagonal operator
\[
 D_u:\ell_s^n\longrightarrow\ell_p^n,
 \qquad
 D_u(z)=(u_jz_j)_{j=1}^n.
\]
H\"older's inequality gives
\begin{equation}\label{whl:eq:Du}
 \|D_u z\|_p\le\|u\|_r\|z\|_s.
\end{equation}
Define
\[
 Q(z):=P(D_u z)
 =\sum_{|\alpha|=m}a_\alpha u^\alpha z^\alpha.
\]
Since $P$ is $m$-homogeneous, \eqref{whl:eq:Du} implies
\[
 \|Q\|_s\le\|u\|_r^m\|P\|_p.
\]
Applying the classical polynomial Hardy--Littlewood inequality \eqref{whl:eq:classical-HL} to $Q$ on $\ell_s^n$ gives
\[
 \left\|(a_\alpha u^\alpha)_{|\alpha|=m}\right\|_{\ell_{q(m,s)}}
 \le
 H^{\mathrm{pol}}_{m,s}(\C)\,\|u\|_r^m\|P\|_p.
\]
Taking the supremum over $\|u\|_r\le1$ gives \eqref{whl:eq:intro-transfer-functional} and proves the transfer assertion in  Theorem \ref{thm:B}.

\begin{corollary}[Exact recovery of the classical Hardy--Littlewood functional]\label{whl:cor:classical-diagonal}
For every $p\ge m$ and every $m$-homogeneous polynomial $P$,
\begin{equation}\label{whl:eq:exact-classical-identity}
 \Wfun_{m,p,p}(P)
 =
 \left\|(a_\alpha)_{|\alpha|=m}\right\|_{\ell_{q(m,p)}}.
\end{equation}
Consequently, the least constant $C$ in
\[
 \Wfun_{m,p,p}(P)\le C\|P\|_p
\]
is exactly $H^{\mathrm{pol}}_{m,p}(\C)$.
\end{corollary}

\begin{proof}
When $s=p$, \eqref{whl:eq:rps-intro} gives $r=\infty$. If $\|u\|_\infty\le1$, then $|u^\alpha|\le1$ for every $|\alpha|=m$, and therefore
\[
 \left\|(a_\alpha u^\alpha)_{|\alpha|=m}\right\|_{\ell_{q(m,p)}}
 \le
 \left\|(a_\alpha)_{|\alpha|=m}\right\|_{\ell_{q(m,p)}}.
\]
Taking the supremum over such $u$ proves one inequality in \eqref{whl:eq:exact-classical-identity}. For the reverse inequality, take $u=(1,\ldots,1)$, which belongs to the unit ball of $\ell_\infty^n$ and satisfies $u^\alpha=1$ for every $\alpha$. This proves \eqref{whl:eq:exact-classical-identity}. The assertion about the optimal constant follows from the definition of $H^{\mathrm{pol}}_{m,p}(\C)$ in \eqref{whl:eq:classical-HL}.
\end{proof}

\subsection{Exact maximum coefficient estimate}
\begin{proposition}[Exact weighted maximum]\label{whl:prop:exact-weighted-maximum}
Let $m\ge1$ and $1\le p,t\le\infty$. The least dimension-free constant $C_{m,p,t}$ for which
\begin{equation}\label{whl:eq:intro-max-general}
 \max_{|\alpha|=m}|a_\alpha u^\alpha|
 \le C_{m,p,t}\,\|u\|_t^m\|P\|_p
\end{equation}
holds for all $P$ and $u$ is
\begin{equation}\label{whl:eq:intro-Cexact}
 C_{m,p,t}=
 \begin{cases}
  m^{m(1/p-1/t)},& p\le t,\\[1mm]
  1,& t\le p.
 \end{cases}
\end{equation}
\end{proposition}
\begin{proof}
By homogeneity, it is enough to assume $\|u\|_t=1$. Applying
\cref{whl:prop:coefficient-functional} and then
\cref{whl:lem:monomial} with exponent $t$, we obtain, for every
$|\alpha|=m$,
\begin{align}
 |a_\alpha u^\alpha|
 &\le
 \left(\frac{m^m}{\alpha^\alpha}\right)^{1/p}
 |u^\alpha|\,\|P\|_p \notag\\
 &\le
 \left(\frac{m^m}{\alpha^\alpha}\right)^{1/p-1/t}
 \|P\|_p.
 \label{whl:eq:alpha-bound}
\end{align}

Suppose first that $p\le t$. Then $1/p-1/t\ge0$, and
\[
 1\le\frac{m^m}{\alpha^\alpha}\le m^m.
\]
The upper value is attained by a square-free multi-index. Therefore
\[
 \max_{|\alpha|=m}|a_\alpha u^\alpha|
 \le
 m^{m(1/p-1/t)}\|P\|_p.
\]
To prove sharpness, take $n=m$,
\[
 P(z)=z_1\cdots z_m,
 \qquad
 u=m^{-1/t}(1,\ldots,1).
\]
Then
\[
 \|u\|_t=1,\qquad
 \|P\|_p=m^{-m/p},
 \qquad
 |u_1\cdots u_m|=m^{-m/t},
\]
and hence the constant cannot be smaller than
$m^{m(1/p-1/t)}$.

Now suppose that $t\le p$. Since $1/p-1/t\le0$ and
$m^m/\alpha^\alpha\ge1$, \eqref{whl:eq:alpha-bound} gives
\[
 \max_{|\alpha|=m}|a_\alpha u^\alpha|
 \le \|P\|_p.
\]
The choice $P(z)=z_1^m$ and $u=e_1$ shows that the constant $1$ is
optimal.
\end{proof}
\begin{corollary}[Unweighted maximum coefficient]\label{whl:cor:unweighted-max}
For $1\le p<\infty$,
\[
 \max_{|\alpha|=m}|a_\alpha|
 \le m^{m/p}\|P\|_p,
\]
and the constant $m^{m/p}$ is optimal. In particular, at $p=m$ one obtains
\[
 H^{\mathrm{pol}}_{m,m}(\C)=m.
\]
\end{corollary}

\subsection{Critical continuation below $p=m$}

Suppose first that $1\le p<m$. With
\[
 r_{m,p}=\frac{mp}{m-p},
\]
one has
\[
 \frac1p-\frac1{r_{m,p}}=\frac1m.
\]
Since $p<r_{m,p}$, Proposition~\ref{whl:prop:exact-weighted-maximum} gives
\[
 \max_{|\alpha|=m}|a_\alpha u^\alpha|
 \le
 m^{m(1/p-1/r_{m,p})}
 \|u\|_{r_{m,p}}^m\|P\|_p
 =m\|u\|_{r_{m,p}}^m\|P\|_p.
\]
The constant is optimal by Proposition~\ref{whl:prop:exact-weighted-maximum}. If $p=m$, take $r_{m,m}=\infty$. The same proposition gives
\[
 \max_{|\alpha|=m}|a_\alpha u^\alpha|
 \le m\|u\|_\infty^m\|P\|_m,
\]
again with optimal constant $m$. This proves the critical continuation assertion in  Theorem \ref{thm:B}.

The critical branch also follows directly from Theorem~\ref{thm:B} by choosing the boundary source $s=m$. In this case $q(m,m)=\infty$ and
\[
 \frac1r=\frac1p-\frac1m,
\]
so $r=mp/(m-p)$ for $p<m$. Proposition~\ref{whl:prop:exact-weighted-maximum} adds the sharpness of the constant.

\begin{corollary}[Attachment at the critical endpoint]\label{whl:cor:attachment}
As $p\uparrow m$,
\[
 \frac{mp}{m-p}\longrightarrow\infty.
\]
Hence the weighted branch in Theorem~\ref{thm:B} reaches
\[
 \max_{|\alpha|=m}|a_\alpha u^\alpha|
 \le m\|u\|_\infty^m\|P\|_m.
\]
Taking $u=(1,\ldots,1)$ gives the classical Hardy--Littlewood endpoint
\[
 \max_{|\alpha|=m}|a_\alpha|
 \le m\|P\|_m.
\]
\end{corollary}

\section{Proof of Theorem \ref{thm:C}: coefficient norm versus sup-norm}\label{sec:proof-C}

The constants $\mathfrak A_{p,r}^m(n)$ defined below are the norms of
the finite-dimensional coefficient embeddings
\[
\mathcal P({}^m\ell_p^n)
\longrightarrow
\ell_r\bigl(\{\alpha\in\mathbb N_0^n:|\alpha|=m\}\bigr).
\]
Bounds for these embedding norms in several parameter ranges appear
in \cite[Theorem~2.1]{GalicerMansillaMuro}. Sparse unimodular
multilinear forms yield lower estimates after their variables are
placed in disjoint coordinate blocks.

For an $m$-homogeneous polynomial
\[
 P(z)=\sum_{|\nu|=m}a_\nu z^\nu
 \qquad (z\in\C^n),
\]
write
\[
 |P|_r:=\left(\sum_{|\nu|=m}|a_\nu|^r\right)^{1/r},
 \qquad
 |P|_\infty:=\max_{|\nu|=m}|a_\nu|,
\]
and let $\mathfrak A_{p,r}^m(n)$ be the least constant such that
\begin{equation}\label{eq:poly-A-def}
 |P|_r\le \mathfrak A_{p,r}^m(n)
 \|P\|_{\mathcal P({}^m\ell_p^n)}
\end{equation}
for every such polynomial.  The symbol $\mathfrak A$ distinguishes the polynomial constant from the multilinear constant $A_{p,r}^m(n)$ in \eqref{eq:Apr-def}.

Fix $2\le p\le m$ and put
\begin{equation}\label{eq:poly-hcDG}
 h:=\lceil p\rceil-1,
 \qquad
 c:=1-\frac hp,
 \qquad
 D:=m-h,
 \qquad
 G:=\frac{2m}{p}-1.
\end{equation}
Thus $D,G\ge1$ and $0<c\le1/2$.  For $2\le r\le\infty$, with the convention $1/\infty=0$, define
\begin{equation}\label{eq:beta-variational}
 \beta_{m,p}(r)
 :=\sup_{1\le d\le m}
 \left\{\frac d r-\alpha_{m,d}(p)\right\}.
\end{equation}

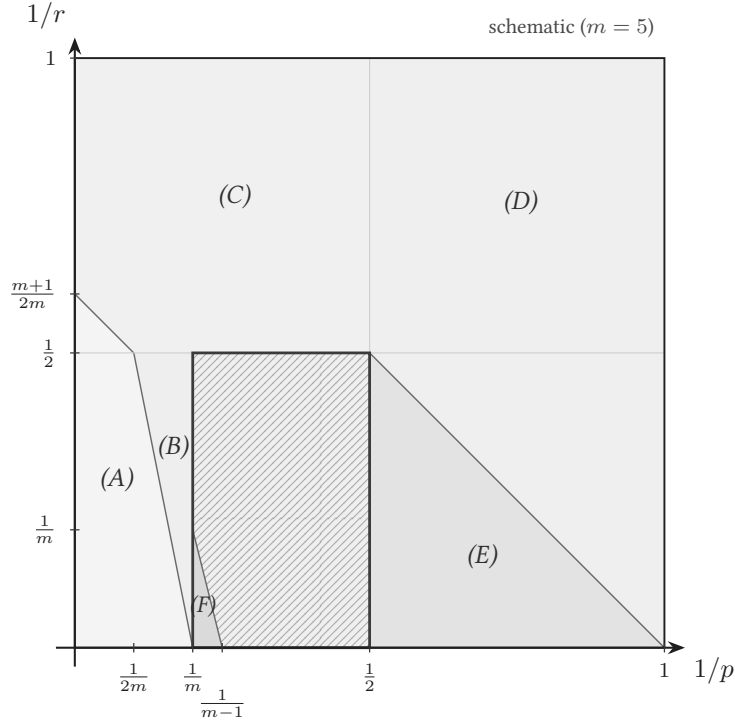
\begin{figure}[H]
\centering
\begin{tikzpicture}[
  x=7.8cm,y=7.8cm,
  >=Stealth,
  every node/.style={font=\footnotesize},
  regionlabel/.style={font=\small\itshape,text=axisInk!82},
  ticklabel/.style={font=\scriptsize,text=axisInk!88}
]
  % The numerical value m=5 fixes only the relative placement of the
  % symbolic breakpoints in this schematic drawing.
  \def\mfig{5}
  \pgfmathsetmacro{\xTwoM}{1/(2*\mfig)}
  \pgfmathsetmacro{\xM}{1/\mfig}
  \pgfmathsetmacro{\xMm}{1/(\mfig-1)}
  \pgfmathsetmacro{\yM}{1/\mfig}
  \pgfmathsetmacro{\yTop}{(\mfig+1)/(2*\mfig)}

  % GMM regions: muted colors, preserving the visual separation of the
  % original phase diagram without competing with the highlighted region.
  \fill[gmmGold!28]   (0,0)--(\xM,0)--(\xTwoM,.5)--(0,\yTop)--cycle;             % A
  \fill[gmmRose!27]   (\xTwoM,.5)--(\xM,0)--(\xM,.5)--cycle;                   % B
  \fill[gmmCoral!20]  (0,\yTop)--(\xTwoM,.5)--(.5,.5)--(.5,1)--(0,1)--cycle;  % C
  \fill[gmmIndigo!17] (.5,.5)--(1,0)--(1,1)--(.5,1)--cycle;                    % D
  \fill[gmmOlive!24]  (.5,0)--(.5,.5)--(1,0)--cycle;                           % E
  \fill[gmmViolet!28] (\xM,0)--(\xM,\yM)--(\xMm,0)--cycle;                   % F

  % Region left open in GMM and determined by Theorem C.
  \fill[ourTeal!9]
      (\xM,\yM)--(\xM,.5)--(.5,.5)--(.5,0)--(\xMm,0)--cycle;
  \fill[pattern=north east lines,pattern color=ourTeal!48]
      (\xM,\yM)--(\xM,.5)--(.5,.5)--(.5,0)--(\xMm,0)--cycle;

  % Structural boundaries.
  \draw[axisInk!58,line width=.55pt] (0,\yTop)--(\xTwoM,.5);
  \draw[axisInk!58,line width=.55pt] (\xTwoM,.5)--(\xM,0);
  \draw[axisInk!58,line width=.55pt] (\xM,\yM)--(\xMm,0);
  \draw[axisInk!58,line width=.55pt] (.5,.5)--(1,0);
  \draw[axisInk!20,line width=.40pt] (0,.5)--(1,.5);
  \draw[axisInk!20,line width=.40pt] (.5,0)--(.5,1);

  % Full strip covered by Theorem C.
  \draw[ourTeal!72!black,line width=1.05pt]
      (\xM,0) rectangle (.5,.5);

  % Labels placed well inside their respective regions for m=5.
  \node[regionlabel] at (.070,.285) {(A)};
  \node[regionlabel] at (.166,.340) {(B)};
  \node[regionlabel] at (.270,.765) {(C)};
  \node[regionlabel] at (.755,.755) {(D)};
  \node[regionlabel] at (.690,.155) {(E)};
  \node[regionlabel,font=\footnotesize\itshape] at (.216,.070) {(F)};

  % A discreet note makes explicit how m=5 enters the drawing.
  \node[anchor=south east,font=\scriptsize,text=axisInk!72]
      at (1,1.018) {schematic ($m=5$)};

  % Frame and axes.
  \draw[axisInk!88,line width=.72pt] (0,0) rectangle (1,1);
  \draw[axisInk!88,line width=.78pt,->] (-.032,0)--(1.035,0)
      node[below right=-1pt,font=\small] {$1/p$};
  \draw[axisInk!88,line width=.78pt,->] (0,-.032)--(0,1.035)
      node[above left=-1pt,font=\small] {$1/r$};

  % Symbolic ticks.  Their positions are drawn with m=5, but the labels are
  % the exact symbolic breakpoints for general m.
  \foreach \x in {\xTwoM,\xM,\xMm,.5,1}
      \draw[axisInk!78,line width=.48pt] (\x,.008)--(\x,-.008);
  \node[ticklabel,below=2pt]  at (\xTwoM,0) {$\frac1{2m}$};
  \node[ticklabel,below=2pt]  at (\xM,0) {$\frac1m$};
  \node[ticklabel,below=12pt] at (\xMm,0) {$\frac1{m-1}$};
  \node[ticklabel,below=2pt]  at (.5,0) {$\frac12$};
  \node[ticklabel,below=2pt]  at (1,0) {$1$};

  \foreach \y in {\yM,.5,\yTop,1}
      \draw[axisInk!78,line width=.48pt] (.008,\y)--(-.008,\y);
  \node[ticklabel,left=3pt] at (0,\yM) {$\frac1m$};
  \node[ticklabel,left=3pt] at (0,.5) {$\frac12$};
  \node[ticklabel,left=3pt] at (0,\yTop) {$\frac{m+1}{2m}$};
  \node[ticklabel,left=3pt] at (0,1) {$1$};
\end{tikzpicture}
\caption{Coefficient-embedding regions in the $(1/p,1/r)$-plane.
The picture is schematic and uses $m=5$ only to locate the displayed
breakpoints.}
\label{fig:GMM-completed-region}
\end{figure}

\begin{proof}
Fix $d\in[1,m]$ and $\eta>0$.  By Theorem~\ref{thm:A}, for all sufficiently large $k$ there is an $m$-linear form
\[
 T_k(x^{(1)},\ldots,x^{(m)})
 =\sum_{\mathbf i\in\Lambda_k}
 \eps_{\mathbf i}x^{(1)}_{i_1}\cdots x^{(m)}_{i_m},
 \qquad |\eps_{\mathbf i}|=1,
\]
with
\begin{equation}\label{eq:poly-app-T-size}
 |\Lambda_k|\ge \lfloor k^d\rfloor,
 \qquad
 \|T_k\|_{\mathcal L({}^m\ell_p^k)}
 \le C_\eta k^{\alpha_{m,d}(p)+\eta}.
\end{equation}
Split $\C^{mk}$ into $m$ consecutive blocks of length $k$ and define
\begin{equation}\label{eq:block-polynomial}
 P_k(z):=T_k(z^{(1)},\ldots,z^{(m)}),
 \qquad
 z=(z^{(1)},\ldots,z^{(m)})\in\C^{mk}.
\end{equation}
Every coefficient of $T_k$ becomes a distinct degree-$m$ monomial of $P_k$, since exactly one variable is taken from each block.  Hence
\begin{equation}\label{eq:block-coeff-norm}
 |P_k|_r=|\Lambda_k|^{1/r}
 \ge \lfloor k^d\rfloor^{1/r}
 \gtrsim_{d,r} k^{d/r},
\end{equation}
with the usual convention $a^0=1$ when $r=\infty$.
If $\|z\|_{\ell_p^{mk}}\le1$ and $a_j:=\|z^{(j)}\|_p$, then
\[
 \sum_{j=1}^m a_j^p\le1.
\]
The arithmetic-geometric mean inequality gives
\[
 \prod_{j=1}^m a_j
 \le\left(\frac1m\sum_{j=1}^m a_j^p\right)^{m/p}
 \le m^{-m/p}.
\]
Therefore
\begin{equation}\label{eq:block-poly-norm}
 \|P_k\|_{\mathcal P({}^m\ell_p^{mk})}
 \le m^{-m/p}\|T_k\|_{\mathcal L({}^m\ell_p^k)}.
\end{equation}
Combining \eqref{eq:poly-app-T-size}--\eqref{eq:block-poly-norm},
\begin{equation}\label{eq:poly-lower-fixed-d}
 \mathfrak A_{p,r}^m(mk)
 \gtrsim_{m,p,r,d,\eta}
 k^{d/r-\alpha_{m,d}(p)-\eta}.
\end{equation}
For arbitrary large $n$, take $k=\lfloor n/m\rfloor$ and extend $P_k$ by zero to $n$ variables.  Since $k\asymp_m n$, \eqref{eq:poly-lower-fixed-d} implies
\[
 \liminf_{n\to\infty}
 \frac{\log\mathfrak A_{p,r}^m(n)}{\log n}
 \ge \frac d r-\alpha_{m,d}(p)-\eta.
\]
Taking first $\eta\downarrow0$ and then the supremum over $d$ gives
\begin{equation}\label{eq:poly-liminf-beta}
 \liminf_{n\to\infty}
 \frac{\log\mathfrak A_{p,r}^m(n)}{\log n}
 \ge\beta_{m,p}(r).
\end{equation}

For the upper bound, let $P$ be $m$-homogeneous on $\ell_p^n$ and let $L:(\ell_p^n)^m\longrightarrow\mathbb C$ be its associated symmetric $m$-linear form.  The coefficient norms of $P$ and $L$, and their supremum and operator norms, are equivalent up to constants depending only on $m$; see \cite[(5)--(6)]{GalicerMansillaMuro}.  Consequently these coefficient estimates transfer to polynomials without changing the power of the dimension.

At $r=\infty$, evaluation of $L$ on basis vectors gives
\begin{equation}\label{eq:poly-endpoint-inf}
 \mathfrak A_{p,\infty}^m(n)\lesssim_m1.
\end{equation}
At $r=2$, Lemma~\ref{lem:l2-critical} with $q_1=\cdots=q_m=2m$, together with
$\|L\|_{\cL({}^m\ell_{2m}^n)}\le n^{m/p-1/2}\|L\|_{\cL({}^m\ell_p^n)}$, gives
\begin{equation}\label{eq:poly-endpoint-two}
 \mathfrak A_{p,2}^m(n)
 \lesssim_{m,p}n^{m/p-1/2}
 =n^{G/2}.
\end{equation}
Finally, set $r_0:=1/c=p/(p-h)$.  Lemma~\ref{lem:Bayart-sectional} gives
\begin{equation}\label{eq:poly-endpoint-sectional}
 \mathfrak A_{p,r_0}^m(n)
 \lesssim_{m,p}n^{(m-h)/r_0}
 =n^{Dc}.
\end{equation}
Interpolation of finite-dimensional coefficient norms now gives two available upper envelopes.  Interpolating \eqref{eq:poly-endpoint-inf} and \eqref{eq:poly-endpoint-two} yields
\begin{equation}\label{eq:poly-upper-global-line}
 \mathfrak A_{p,r}^m(n)\lesssim_{m,p,r}n^{G/r}
 \qquad(2\le r\le\infty).
\end{equation}
For $0\le x=1/r\le c$, interpolation between $r_0$ and $\infty$ gives
\begin{equation}\label{eq:poly-upper-section-line1}
 \mathfrak A_{p,r}^m(n)\lesssim n^{Dx},
\end{equation}
while for $c\le x\le1/2$, interpolation between $r_0$ and $2$ gives
\begin{equation}\label{eq:poly-upper-section-line2}
 \mathfrak A_{p,r}^m(n)
 \lesssim
 n^{Dc+\frac{x-c}{1/2-c}(G/2-Dc)}
\end{equation}
when $c<1/2$.  If $p=2$, then $c=1/2$ and the two finite endpoints coincide.

It remains to evaluate \eqref{eq:beta-variational}.\begin{equation}\label{eq:alpha-DG-form}
 \alpha_{m,d}(p)
 =\max\left\{0,\frac{d-G}{2},c(d-D)_+\right\}.
\end{equation}
Suppose first that $D\ge G$.  Since $c\le1/2$, the sectional affine function never exceeds $(d-G)/2$ once it becomes positive.  Thus $\alpha_{m,d}(p)=0$ for $1\le d\le G$ and $\alpha_{m,d}(p)=(d-G)/2$ for $G\le d\le m$.  As $x=1/r\le1/2$, the function
\[
 d\longmapsto dx-\alpha_{m,d}(p)
\]
is increasing up to $d=G$ and nonincreasing afterwards.  Hence
\[
 \beta_{m,p}(r)=Gx=\frac Gr,
\]
which agrees with the upper bound \eqref{eq:poly-upper-global-line}.

Assume now that $D<G$.  Then $p>2$, so $c<1/2$.  The sectional line starts at $d=D$, the global line starts at $d=G$, and they meet at the value $d_*$ in \eqref{eq:dstar-def}.  Moreover
\begin{equation}\label{eq:dstar-range}
 D<G<d_*\le m.
\end{equation}
Indeed, using $d_*=(G-2cD)/(1-2c)$ and $0<c<1/2$, we have
\begin{align*}
 d_*-G
 &=\frac{G-2cD-G(1-2c)}{1-2c}\\
 &=\frac{2c(G-D)}{1-2c}>0,
\end{align*}
so $d_*>G$.  For the upper bound, $d_*\le m$ is equivalent to
\[
 G-2cD\le m(1-2c).
\]
Substituting
\[
 G=\frac{2m}{p}-1,\qquad
 c=1-\frac hp,\qquad
 D=m-h,
\]
and multiplying by $p>0$, this becomes
\[
 p(m+1-2h)+2h^2-2m\ge0.
\]
Set
\[
 E(p):=p(m+1-2h)+2h^2-2m.
\]
Since $h=\lceil p\rceil-1$, we have $h\le p\le h+1$; moreover $D<G$ implies $p>2$, hence $h\ge2$.  The function $E$ is affine on $[h,h+1]$, and its endpoint values are
\begin{align*}
 E(h)&=h(m+1)-2m=m(h-2)+h\ge0,\\
 E(h+1)&=(h-1)(m-1)\ge0.
\end{align*}
Therefore $E(p)\ge0$ throughout $[h,h+1]$, so $d_*\le m$.  Hence
\begin{equation}\label{eq:alpha-three-piece}
 \alpha_{m,d}(p)=
 \begin{cases}
 0,&1\le d\le D,\\
 c(d-D),&D\le d\le d_*,\\
 (d-G)/2,&d_*\le d\le m.
 \end{cases}
\end{equation}
The slopes of $dx-\alpha_{m,d}(p)$ on these three intervals are, respectively,
\[
 x,\qquad x-c,\qquad x-\frac12.
\]
If $0\le x\le c$, the maximum is attained at $d=D$, and therefore $\beta_{m,p}(r)=Dx$.  If $c\le x\le1/2$, the maximum is attained at $d=d_*$, and
\begin{align*}
 \beta_{m,p}(r)
 &=d_*x-c(d_*-D)\\
 &=Dc+\frac{x-c}{1/2-c}
 \left(\frac G2-Dc\right),
\end{align*}
where the second identity uses
\[
 c(d_*-D)=\frac{d_*-G}{2}.
\]
This proves \eqref{eq:beta-case-DltG}.  The upper bounds \eqref{eq:poly-upper-section-line1}--\eqref{eq:poly-upper-section-line2} now match \eqref{eq:poly-liminf-beta}, and \eqref{eq:poly-A-limit} follows.
\end{proof}

\begin{remark}\label{rem:Steiner-comparison}
For example, let $m=5$, $p=3$ and $r=4$.  Then $G=7/3$ and $D=3$, so Theorem~\ref{thm:C} gives
\[
 \mathfrak A_{3,4}^5(n)=n^{7/12+o(1)}.
\]
The maximizing support-growth dimension is $d=G=7/3$, for which $\alpha_{5,7/3}(3)=0$.  By comparison, a partial Steiner support has order $n^{m-1}=n^4$ coefficients and the estimate of \cite[Theorem~1.5]{GalicerMansillaMuro} gives, up to a logarithmic factor, norm $n^{5(1/2-1/3)}=n^{5/6}$.  At $r=4$ this produces only the lower power $n^{1/6}$.  Thus allowing the support size to have a fractional exponent changes the extremal lower bound from $1/6$ to the sharp value $7/12$.
\end{remark}

\section{Proof of Theorem \ref{thm:D}: sharp diagonal weight exponents}\label{sec:proof-D}

\subsection{Finite-exponent source branches}

The source parameter $s$ also produces finite coefficient exponents below the critical ambient value $p=m$.

\begin{corollary}[Lower Hardy--Littlewood source range]\label{whl:cor:lower-source}
Let $m<s<2m$ and $1\le p\le s$. Then
\begin{equation}\label{whl:eq:lower-source}
 \left(\sum_{|\alpha|=m}|a_\alpha u^\alpha|^{\frac{s}{s-m}}\right)^{\frac{s-m}{s}}
 \le
 H^{\mathrm{pol}}_{m,s}(\C)
 \|u\|_{\frac{ps}{s-p}}^m\|P\|_p,
\end{equation}
with the convention $ps/(s-p)=\infty$ when $p=s$.
\end{corollary}

\begin{proof}
Apply Theorem~\ref{thm:B} and use $q(m,s)=s/(s-m)$.
\end{proof}

\begin{corollary}[Quadratic source]\label{whl:cor:l2}
Let $1\le p\le2m$. Then
\begin{equation}\label{whl:eq:l2-branch}
 \left(\sum_{|\alpha|=m}|a_\alpha u^\alpha|^2\right)^{1/2}
 \le
 H^{\mathrm{pol}}_{m,2m}(\C)
 \|u\|_{\frac{2mp}{2m-p}}^m\|P\|_p,
\end{equation}
with the usual convention at $p=2m$.
\end{corollary}

\begin{proof}
Take $s=2m$ in Theorem~\ref{thm:B}.
\end{proof}

\begin{corollary}[Bohnenblust--Hille source]\label{whl:cor:BH}
Let $1\le p\le\infty$. Then
\begin{equation}\label{whl:eq:BH-branch}
 \left(\sum_{|\alpha|=m}|a_\alpha u^\alpha|^{\frac{2m}{m+1}}\right)^{\frac{m+1}{2m}}
 \le
 D_m\|u\|_p^m\|P\|_p,
\end{equation}
where $D_m=H^{\mathrm{pol}}_{m,\infty}(\C)$.
\end{corollary}

\begin{proof}
Take $s=\infty$ in Theorem~\ref{thm:B}.
\end{proof}

For fixed $p<m$, increasing $s$ from $m$ to $\infty$ moves continuously through the classical source scale: $s=m$ gives the maximum norm, $s=2m$ gives the coefficient $\ell_2$ norm, and $s=\infty$ gives the Bohnenblust--Hille exponent. The corresponding diagonal weight is always determined by
\[
 \frac1r=\frac1p-\frac1s.
\]

\subsection{The strip $1\le q\le2$}

We first record the unimodular polynomials used for necessity. A result of Bayart \cite{Bayart2012}, in the form recalled in \cite{GalicerMansillaMuro}, gives the following: for each $m\ge2$, $2\le p\le\infty$, and $n\ge1$, there is an $m$-homogeneous polynomial
\begin{equation}\label{whl:eq:unimodular-poly}
 U_n(z)=\sum_{|\alpha|=m}\varepsilon_\alpha z^\alpha,
 \qquad |\varepsilon_\alpha|=1,
\end{equation}
with
\begin{equation}\label{whl:eq:unimodular-norm}
 \|U_n\|_p
 \le K_{m,p}\,
 n^{m(\frac12-\frac1p)+\frac12},
\end{equation}
where $K_{m,p}$ is independent of $n$.

We prove sufficiency in two ranges. Suppose first that $q_m\le q\le2$. There is a unique $s\in[2m,\infty]$ such that
\begin{equation}\label{whl:eq:s-from-q-upper}
 \frac1q=\frac{m+1}{2m}-\frac1s.
\end{equation}
Since $p\le m\le s$, Theorem~\ref{thm:B} applies. Moreover,
\[
 \frac1p-\frac1s
 =\frac1p+\frac1q-\frac{m+1}{2m}
 =\frac1{t_*},
\]
so \eqref{whl:eq:sharp-strip-intro} follows.

Now let $1\le q<q_m$ and choose $h\in(0,\infty)$ by
\begin{equation}\label{whl:eq:h-factor}
 \frac1q=\frac1{q_m}+\frac1h.
\end{equation}
Equation \eqref{whl:eq:tstar-intro} becomes
\[
 \frac1{t_*}=\frac1p+\frac1h.
\]
Fix $u\in\C^n$. Replacing $u$ by $|u|$ does not change the left-hand side. The standard coordinatewise factorization $\ell_{t_*}=\ell_p\cdot\ell_h$ gives nonnegative vectors $v,w$ such that
\[
 |u_j|=v_jw_j,
 \qquad
 \|v\|_p\|w\|_h=\|u\|_{t_*}.
\]
For instance, when $u\neq0$ one may take
\[
 v_j=|u_j|^{t_*/p},
 \qquad
 w_j=|u_j|^{t_*/h},
\]
and then $\|v\|_p\|w\|_h=\|u\|_{t_*}$. H\"older's inequality in the coefficient index, \eqref{whl:eq:BH-branch}, and the multinomial theorem give
\begin{align*}
 \left(\sum_{|\alpha|=m}|a_\alpha u^\alpha|^q\right)^{1/q}
 &\le
 \left(\sum_{|\alpha|=m}|a_\alpha v^\alpha|^{q_m}\right)^{1/q_m}
 \left(\sum_{|\alpha|=m}|w^\alpha|^h\right)^{1/h}\\
 &\le
 D_m\|v\|_p^m\|P\|_p\,\|w\|_h^m\\
 &=D_m\|u\|_{t_*}^m\|P\|_p.
\end{align*}
Here
\[
 \sum_{|\alpha|=m}|w^\alpha|^h
 \le \left(\sum_{j=1}^n|w_j|^h\right)^m.
\]
This proves sufficiency for all $1\le q\le2$.

For necessity, suppose that for some $t\in[1,\infty]$ and $C>0$ independent of $n$,
\begin{equation}\label{whl:eq:hyp-sharp-strip}
 \left(\sum_{|\alpha|=m}|a_\alpha u^\alpha|^q\right)^{1/q}
 \le C\|u\|_t^m\|P\|_p
\end{equation}
holds for every $n,P,u$. Apply it to $U_n$ from \eqref{whl:eq:unimodular-poly} and to
\[
 u^{(n)}=n^{-1/t}(1,\ldots,1),
\]
with the usual interpretation when $t=\infty$. Since
\[
 \#\{\alpha\in\Nzero^n:|\alpha|=m\}
 =\binom{n+m-1}{m}\asymp_m n^m,
\]
we obtain
\[
 n^{m/q-m/t}
 \lesssim_{m,p,q}
 n^{m(\frac12-\frac1p)+\frac12}.
\]
Hence
\[
 \frac1t
 \ge
 \frac1p+\frac1q-\frac{m+1}{2m}
 =\frac1{t_*}.
\]
Thus $t\le t_*$, and part~\textup{(i)} of Theorem~\ref{thm:D} is proved.

\begin{remark}\label{whl:rem:q1-monomial}
The endpoint $q=1$ belongs to the broader theory of coefficient
multipliers and monomial convergence for homogeneous polynomials and
holomorphic functions; see
\cite{BayartDefantFrerickMaestreSevillaPeris2017,
BayartDefantSchlueters2019,
GalicerMansillaMuroSevillaPeris2021}.
In that literature the central question is typically absolute monomial
convergence, whereas Theorem~\ref{thm:D} places this endpoint inside a finite
$\ell_q$ scale and asks for the sharp dimension-free $\ell_t$
condition on the diagonal weight.
\end{remark}

\subsection{The critical line $p=m$}

For $1\le q\le2$, part~\textup{(ii)} follows directly from
part~\textup{(i)} with $p=m$.

Let $2\le q<\infty$ and set
\[
 s:=\frac{mq}{q-1}.
\]
Then $m<s\le2m$ and
\[
 q=\frac{s}{s-m}.
\]
Theorem~\ref{thm:B} with ambient exponent $p=m$ gives
\[
 \left\|(a_\alpha u^\alpha)_{|\alpha|=m}\right\|_{\ell_q}
 \le
 H^{\mathrm{pol}}_{m,s}(\C)
 \|u\|_{mq}^m\|P\|_m,
\]
because
\[
 \frac1m-\frac1s=\frac1{mq}.
\]
For $q=\infty$, Proposition~\ref{whl:prop:exact-weighted-maximum} with $p=m$ and $t=\infty$ gives the required estimate.

It remains to prove optimality of the weight exponent for $q\ge2$. Let
\[
 P_n(z):=\sum_{j=1}^n z_j^m.
\]
Since $\|P_n\|_m=1$, applying a hypothetical dimension-free estimate with weight exponent $t$ to $u^{(n)}=n^{-1/t}(1,\ldots,1)$ gives
\[
 n^{1/q-m/t}\le C.
\]
Therefore $1/t\ge1/(mq)$, or $t\le mq$. For $q=\infty$ this places no restriction, exactly as \eqref{whl:eq:rho-critical-intro} prescribes. This proves part~\textup{(ii)} of Theorem~\ref{thm:D}.

\subsection{The region $q>2$}

The direct source-transfer curve from Theorem~\ref{thm:B} is not the strongest estimate available when $p<m$ and $q>2$. The quadratic branch and the exact maximum-coefficient estimate can be combined without any interpolation theorem for polynomial spaces.

\begin{proposition}[Quadratic-to-maximum interpolation]\label{whl:prop:qgt2}
Let $m\ge2$, $2\le p\le m$, and $2\le q\le\infty$. Put
\[
 \theta:=\frac2q
\]
with $\theta=0$ for $q=\infty$, and define $\sigma=\sigma(m,p,q)$ by
\begin{equation}\label{whl:eq:sigma-qgt2}
 \frac1\sigma
 =\frac2q\left(\frac1p-\frac1{2m}\right).
\end{equation}
Then
\begin{equation}\label{whl:eq:qgt2-interp}
 \left\|(a_\alpha u^\alpha)_{|\alpha|=m}\right\|_{\ell_q}
 \le
 \bigl(H^{\mathrm{pol}}_{m,2m}(\C)\bigr)^\theta
 m^{\frac{m}{p}(1-\theta)}
 \|u\|_\sigma^m\|P\|_p.
\end{equation}
For $p<m$ and $q>2$, \eqref{whl:eq:qgt2-interp} allows a strictly larger weight exponent than the direct source-transfer estimate with the same coefficient exponent. For finite $q>2$, the exponent $\sigma$ is optimal when $p=2$ and when $p=m$. At $q=\infty$, one has $\sigma=\infty$, and optimality holds for every $2\le p\le m$ by Proposition~\ref{whl:prop:exact-weighted-maximum}.
\end{proposition}

\begin{proof}
For $q=2$, the result is exactly \cref{whl:cor:l2}. If
$q=\infty$, then $\sigma=\infty$ and
\[
 \max_{|\alpha|=m}|a_\alpha u^\alpha|
 \le
 \|u\|_\infty^m
 \max_{|\alpha|=m}|a_\alpha|
 \le
 m^{m/p}\|u\|_\infty^m\|P\|_p
\]
by \cref{whl:cor:unweighted-max}. Thus the claimed estimate also holds
at $q=\infty$.

Assume from now on that $2<q<\infty$, and let $\theta=2/q$. Set
\[
 t_2:=\frac{2mp}{2m-p},
 \qquad
 \frac1{t_2}=\frac1p-\frac1{2m}.
\]
Then \eqref{whl:eq:sigma-qgt2} is $1/\sigma=\theta/t_2$, hence $\sigma=t_2/\theta$. Replacing $u$ by $|u|$, put
\[
 v_j:=|u_j|^{1/\theta}.
\]
Then $|u^\alpha|=(v^\alpha)^\theta$ and
\[
 \|v\|_{t_2}^{\theta}=\|u\|_\sigma.
\]
For
\[
 A_\alpha:=|a_\alpha|v^\alpha,
 \qquad
 B_\alpha:=|a_\alpha|,
\]
we have $|a_\alpha u^\alpha|=A_\alpha^\theta B_\alpha^{1-\theta}$. H\"older's inequality gives
\[
 \|(a_\alpha u^\alpha)\|_{\ell_q}
 \le \|A\|_{\ell_2}^{\theta}\|B\|_{\ell_\infty}^{1-\theta}.
\]
Now use \eqref{whl:eq:l2-branch} for $A$ and \cref{whl:cor:unweighted-max} for $B$. This proves \eqref{whl:eq:qgt2-interp}.

When $p=m$, \eqref{whl:eq:sigma-qgt2} gives $1/\sigma=1/(mq)$, so optimality follows from Theorem~\ref{thm:D}\textup{(ii)}. When $p=2$, the formula gives
\[
 \frac1\sigma=\frac{m-1}{mq}.
\]
The partial Steiner-system construction of Galicer, Muro and
Sevilla-Peris \cite{GalicerMuroSevillaPeris2018} provides
$m$-homogeneous unimodular polynomials with approximately $n^{m-1}$
active coefficients and small norm; see also
\cite[Section~5.1]{BayartJEMS} and \cite{GalicerMansillaMuro}.
In particular, there are $m$-homogeneous polynomials $S_n$ with at
least $c_m n^{m-1}$ unimodular coefficients and
\[
 \|S_n\|_2\le C_m(\log n)^{3/2}.
\]
If an estimate of the form \eqref{whl:eq:qgt2-interp} held with a larger weight exponent $t>\sigma$, then applying it to $S_n$ and $u^{(n)}=n^{-1/t}(1,\ldots,1)$ would give
\[
 n^{(m-1)/q-m/t}
 \lesssim_m (\log n)^{3/2},
\]
which is impossible because $(m-1)/q-m/t>0$. Thus $t\le\sigma$.
\end{proof}

For comparison, when $2<q<\infty$ the direct source choice $m<s<2m$ gives
\[
 \frac1{t_{\mathrm{src}}}
 =\frac1p-\frac1m+\frac1{mq}.
\]
A direct calculation shows
\[
 \frac1{t_{\mathrm{src}}}-\frac1\sigma
 =\left(1-\frac2q\right)\left(\frac1p-\frac1m\right).
\]
Thus $\sigma>t_{\mathrm{src}}$ whenever $p<m$ and $q>2$, while the two exponents agree on the critical line $p=m$.

\begin{proposition}[Sparse-KSZ obstruction for diagonal weights]\label{whl:prop-ksz-obstruction}
Let $m\ge2$, $2\le p\le m$, and $2\le q\le\infty$. Suppose that for some $t\in[1,\infty]$ there is a constant $C$, independent of the dimension, such that
\begin{equation}\label{whl:eq-weighted-general-q}
 \left\|(a_\alpha u^\alpha)_{|\alpha|=m}\right\|_{\ell_q}
 \le C\|u\|_t^m\|P\|_p
\end{equation}
for every $m$-homogeneous polynomial $P$ and every $u$. Then
\begin{equation}\label{whl:eq-ksz-necessary}
 \frac1t\ge \frac{\beta_{m,p}(q)}m,
\end{equation}
where $\beta_{m,p}(q)$ is the exponent determined in Theorem~\ref{thm:C}.
\end{proposition}

\begin{proof}
For dimension $N$, take $u^{(N)}=N^{-1/t}(1,\ldots,1)$, with the usual interpretation when $t=\infty$. Since $|\alpha|=m$, one has $(u^{(N)})^\alpha=N^{-m/t}$. Hence \eqref{whl:eq-weighted-general-q} implies
\[
 |P|_q\le C N^{m/t}\|P\|_p.
\]
Thus $\mathfrak A_{p,q}^m(N)\le C N^{m/t}$. Theorem~\ref{thm:C} gives $\mathfrak A_{p,q}^m(N)=N^{\beta_{m,p}(q)+o(1)}$, and therefore \eqref{whl:eq-ksz-necessary} follows after taking logarithms and dividing by $\log N$.
\end{proof}

\begin{corollary}[Sharp part of the region $q>2$]\label{whl:cor-qgt2-sharp}
Let $m\ge2$, $2\le p\le m$, $2\le q\le\infty$, and put
\[
 h:=\lceil p\rceil-1,\qquad D:=m-h,\qquad G:=\frac{2m}{p}-1.
\]
If $D\ge G$, then the exponent $\sigma$ in \eqref{whl:eq:sigma-qgt2} is optimal. Equivalently,
\[
 \frac1{t_{\rm crit}}=\frac{G}{mq}
 =\frac2q\left(\frac1p-\frac1{2m}\right).
\]
\end{corollary}

\begin{proof}
The upper bound is Proposition~\ref{whl:prop:qgt2}. If $D\ge G$, Theorem~\ref{thm:C} gives $\beta_{m,p}(q)=G/q$. Proposition~\ref{whl:prop-ksz-obstruction} forces
\[
 \frac1t\ge\frac{G}{mq}=\frac1\sigma,
\]
so no exponent larger than $\sigma$ is admissible.
\end{proof}

\section{Weighted coefficient estimates in the unresolved range}
\label{sec:remaining}

Let
\[
2<p<m,\qquad 2<q<\infty,\qquad D<G.
\]
Suppose that, for some $t\in[1,\infty]$ and a constant $C$
independent of $n$,
\[
\left(
\sum_{|\alpha|=m}|a_\alpha u^\alpha|^q
\right)^{1/q}
\le
C\|u\|_t^m\|P\|_p
\]
holds for every $m$-homogeneous polynomial $P$ on $\ell_p^n$
and every $u\in\ell_t^n$. Taking
\[
u=n^{-1/t}(1,\ldots,1),
\]
with the usual interpretation for $t=\infty$, gives
\[
\left\|(a_\alpha)_{|\alpha|=m}\right\|_{\ell_q}
\le Cn^{m/t}\|P\|_p.
\]
Theorem~\ref{thm:C} therefore implies the necessary condition
\[
\frac1t\ge\frac{\beta_{m,p}(q)}m.
\]

Proposition~\ref{whl:prop:qgt2} gives a dimension-free weighted
coefficient estimate for the exponent determined by
\[
\frac1t
=
\frac2q\left(\frac1p-\frac1{2m}\right).
\]
These necessary and sufficient conditions do not identify the
largest admissible weight exponent throughout the stated parameter
range. The growth exponent of the unweighted coefficient embedding
is given by Theorem~\ref{thm:C}.

\section*{Use of artificial intelligence}
Generative artificial intelligence tools were used during the preparation of this manuscript to assist with exposition, organization, and the checking of intermediate calculations. The authors reviewed the resulting material and take full responsibility for the mathematical content and the final text.

\medskip

\section*{Funding}
D. Núñez-Alarcón, A. Raposo, and J. Santos are supported by the Conselho Nacional de Desenvolvimento Científico e Tecnológico (CNPq, Brazil). D. Núñez-Alarcón, A. Raposo, and J. Santos hold Grants No. 406457/2023-9 (CNPq/MCTI No. 10/2023) and No. 403964/2024-5 (MCTI/CNPq No. 16/2024). In addition, A. Raposo and J. Santos are supported by Grants No. 302341/2025-0 and No. 305655/2025-6, respectively.

\medskip

\section*{Competing interests}
The authors declare that they have no competing interests.

\section*{Data availability}
No data were used in this work.

\end{document}